\documentclass[12pt]{amsart}

\usepackage{amssymb,epsfig,amsfonts,setspace,dsfont, xcolor}
\usepackage{amsmath,amscd,amsthm,array,geometry}

\newtheorem{thm}{Theorem}[section]

\newtheorem{lem}[thm]{Lemma}
\newtheorem{prop}[thm]{Proposition}
\newtheorem{conj}[thm]{Conjecture}
\newtheorem{cor}[thm]{Corollary}

\theoremstyle{definition}
\newtheorem{defn}[thm]{Definition}
\newtheorem{rem}[thm]{Remark}
\newtheorem{exam}[thm]{Example}

\newcommand{\G}{\mathbf{G}}

\newcommand{\vol}{\mathrm{vol}}

\newcommand{\tr}{\mathrm{tr}}
\newcommand{\Hom}{\mathrm{Hom}}

\newcommand{\adele}{\mathbb{A}}
\newcommand{\finadele}{\mathbb{A}^{\mathrm{fin}}}

\newcommand{\RR}{\mathbb{R}}
\newcommand{\CC}{\mathbb{C}}

\newcommand{\FF}{\mathbb{F}}

\newcommand{\GL}{\mathbf{GL}}

\newcommand{\frob}{\mathrm{Frob}}
\newcommand{\End}{\mathrm{End}}
\newcommand{\NS}{\mathrm{NS}}
\newcommand{\Pic}{\mathrm{Pic}}

\begin{document}

    \title{How large is $A_g(\mathbb{F}_q)$?}
\author{Michael Lipnowski}
\author{Jacob Tsimerman}
\maketitle

\setcounter{section}{-1}


\section{Introduction}

A very natural question in arithmetic statistics is the following: \emph{Let $N(p,g)$ denote the number of smooth, projective curves over the finite field $\FF_p$ of genus $g$. What are the asymptotics of $N(p,g)$ for fixed $p$ and $g\rightarrow\infty$?} From a modern standpoint, one can rephrase this question as asking for the number of $\FF_p$ points on the moduli stack $M_g$ of genus $g$ curves.  A naive analysis using the Weil conjectures, the most powerful tool available for counting points on varieties over finite fields, falls far short of answering this question owing to the ``explosion in topological complexity " of $M_g(\CC)$ as $g$ grows. See \S \ref{topologicalexplosion} for further disucssion.  In particular, the sharpest known upper and lower bounds for $\log N(p,g)$ aren't even of the same magnitude! The lower bound is linear in $g$, while the upper bound is linear in $g \log g$ \cite{DejongKatz}.

In this paper we make a detailed study of an analogous question: \emph{Let $A(p,g)$ denote the number of principally polarized abelian varieties (ppavs) of dimension $g$ over the finite field $\FF_p$.  What are the asymptotics of $A(p,g)$ for $p$ fixed and $g \rightarrow \infty$?}  The moduli stacks $A_g$ similarly experience topological explosion as $g$ grows.  However, a direct analysis of $A_g(\mathbb{F}_p)$ is possible thanks to deep theorems of Honda \cite{Honda} and Tate \cite{Tate1}: ppavs over finite fields are parametrized by linear algebraic data\footnote{In a different vein, this parametrization has had spectacular applications to the computation of zeta functions of Shimura varieties \cite{Langlands} \cite{Kottwitz}.}.
Carefully analyzing this linear parametrization, we prove that $\log A(p,g)$ grows at least as quickly as a multiple of $g^2\log g$. It seems likely to us that this is the correct order of magnitude, but we refer the reader to \S \ref{topologicalexplosion} for further discussion.  Surprisingly, we also show that if one forgets the polarization, and simply counts the number $B(p,g)$ of abelian varieties over $\FF_p$, the rate of growth of $\log B(p,g)$ is much slower.  In particular,  we prove

\begin{thm}\label{fewunpolarized}
Let $B(p,g)$  be the number of isomorphism classes of abelian varieties over $\FF_p$ of dimension $g$. Then 
$$B(p,g) = o_{\epsilon} \left(p^{(\frac{17}{2}+\epsilon)g^2}\right)$$.
\end{thm}

In light of Theorem \ref{fewunpolarized}, this means that the abundance of principal polarizations on a fixed abelian variety is the primary reason for $A(p,g)$ being so large. Following this line of reasoning leads to some statistically counterintuitive behaviour, stemming from the fact that - at least intuitively - the reducible abelian varieties seem to be the ones with the most polarizations.

As it becomes difficult in general to control the number of principal polarizations on an abelian variety, the most striking results require assuming an (to us, plausible) assumption (Conjecture \ref{mrpiconjecture}) about counting principal polarizations on a given abelian variety.  Assuming this conjecture, we show that most ppavs correspond to abelian varieties which, up to isogeny, have an extremely large factor which is just a power of an elliptc curve. In particular, we have the following:

\begin{thm}\label{powerelliptic}
Let $p$ be a prime satisfying the conclusion of Lemma \ref{goodellipticcurves}.  Conjecture \ref{mrpiconjecture} implies that the proportion of principally polarized abelian varieties over $\mathbb{F}_p$ for which $A$ admits an isogeny factor $E^h$ for some elliptic curve $E$ and some $h \geq 0.99g$ approaches 1 as $g$ grows.
\end{thm}

 Moreover, we can use this to show that ppavs don't obey the famous Cohen-Lenstra heuristics (see Corollary \ref{clfails}). Unconditionally, we can prove weaker versions of the above. For instance, we can prove that most ppavs correspond to abelian varieties with a repeated isogeny factor:
 
 \begin{thm}\label{nonsquarefree}
Suppose the prime $p$ satisfies the conclusion of Lemma \ref{goodellipticcurves}.  The probability that a principally polarized abelian variety over the prime field $\mathbb{F}_p$ has no repeated isogeny factors and whose Frobenius characteristic polynomial is relatively prime to $x^2 - p$ approaches 0 as $g$ increases.
\end{thm}

The paper is organized as follows: In section 1 we briefly explain the approach to such counting questions using the Weil conjectures, and why it fails in cases where the prime in question is fixed and the topological complexity of the algebraic varieties in question grow very quickly. Sections 2 and 3 are devoted to the proof of Theorem \ref{fewunpolarized}. Section 2 bounds the number of isogeny classes, while section 3 shows bound the number of isomorphism classes in each isogeny classs. The argument has two parts: First, there is a  local aspect, where we bound the number of modules under a very non-maximal local ring. This is the most difficult part, and here we use a key insight of Yun. Second, there is an adelic global aspect, where we bound the number of ways to `glue' the different structures for each prime. Section 4 brings polarizations into the picture, and shows that (at least for the majority of primes $p$) $A(p,g)$ grows superexponentially in $g^2$, and then uses this to prove theorem \ref{nonsquarefree} as well as some other surprising corollaries.
Finally, section 5 is more speculative. Here we show that if we assume Conjecture 5.2, we can prove the surprising Theorem \ref{powerelliptic}, and violate several natural behaviors, including analogues of the Cohen-Lenstra heuristics and the Katz-Sarnak Heuristics.

\subsection{Acknowledgements}
The authors thank Peter Sarnak and Akshay Venkatesh for their encouragement.  They also thank Gopal Prasad and J.K. Yu for helpful comments on the Siegel mass formula.

\section{Topological explosion and point counting}
\label{topologicalexplosion}
This section can safely be skipped by experts.  It informally explains some difficulties inherent in counting points on $M_g$ and $A_g,$ and is not used in the rest of the paper.  Attempts to overcome topological explosion in $A_g$ and prove ``random-matrix-quality-cancellation" in the Grothendieck-Lefschetz trace formula were the impetus for this paper.

Let $V / \mathbb{F}_p$ be an algebraic variety.  By the Grothendieck-Lefschetz trace formula,
\begin{equation} \label{grothendiecklefschetz}
\# V(\mathbb{F}_p) = \sum (-1)^i \tr \left( \frob_p | H^i_{et,c}(V_{\overline{\mathbb{F}_p}}, \mathbb{Q}_\ell) \right)
\end{equation}

for any prime $\ell \neq p = \mathrm{char}(\mathbb{F}_q).$  Deligne's Riemann hypothesis implies that the eigenvalues of $\frob_q$ acting on $H^i_{et,c}(V_{\overline{\mathbb{F}_p}}, \mathbb{Q}_\ell),$ a priori lying in $\overline{\mathbb{Q}_\ell},$ are algebraic integers all of whose conjugates have absolute value $q^{\frac{w}{2}}$ for some $w \leq i.$  If $V / \mathbb{F}_p$ is smooth and projective, all of these eigenvalues have absolute value exactly equal to $q^{\frac{i}{2}}.$

Remarkably, \'{e}tale cohomology groups are related to singular cohomology of complex varieties.  In particular, if $\mathcal{V}$ is smooth\footnote{We also assume that $\mathcal{V}/\mathbb{Z}_{(p)}$ is the complement in some $\overline{\mathcal{V}} / \mathbb{Z}_{(p)}$ which is smooth and proper of a relative normal crossings divisor.  This assumption is satisfied for $M_g / \mathbb{Z}_{(p)}$ \cite{DM} and $A_g / \mathbb{Z}_{(p)}$ \cite{FC}.} over $\mathbb{Z}_{(p)}$ with $\mathcal{V}_{\mathbb{F}_p} = V,$ 

\begin{equation} \label{bettinumbers}
\dim_{\mathbb{Q}_\ell} H^i_{et,c}(V_{\overline{\mathbb{F}_p}}, \mathbb{Q}_\ell) = \dim_{\mathbb{Q}} H^i_{\mathrm{betti},c}(\mathcal{V}(\CC), \mathbb{Q}),
\end{equation}

\begin{equation} \label{connectedcomponents}
b_i(V) := \dim H^i_{et,c}(V_{\overline{\mathbb{F}_p}}, \mathbb{Q}_\ell) = \# \text{ connected components of } V_{\overline{\mathbb{F}_p}},
\end{equation}

\begin{equation} \label{frobeniusfundamentalclass}
\frob_p \text{ acts on } H^{2\dim V}(V_{\overline{\mathbb{F}_p}}, \mathbb{Q}_\ell) \text{ as multiplication by } q^{\dim V}.
\end{equation}

Combining \eqref{grothendiecklefschetz}, \eqref{bettinumbers}, \eqref{connectedcomponents}, and \eqref{frobeniusfundamentalclass} shows that if $V / \mathbb{F}_p$ is smooth and geometrically connected, 

\begin{equation} \label{numberofpoints}
\# V(\mathbb{F}_p) = p^{\dim V} + \sum_{i < 2 \dim V} S_i(V),
\end{equation}

where every $S_i(V)$ is a sum of $b_i(V)$ complex numbers of absolute value $p^{\frac{w}{2}}$ for some $w \leq i.$  Foregoing all cancellation in the sum \eqref{numberofpoints} gives the \emph{Weil bound}:

\begin{equation} \label{weilbound}
\# V(\mathbb{F}_p) = p^{\dim V} + E \text{ where } |E| \leq p^{\frac{\dim V - 1}{2}} b(V) \text{ and } b(V) := \sum_{i = 0}^{2\dim V - 1} b_i(V)
\end{equation}

which leads one to suspect that 

\begin{equation} \label{naivepointcountguess}
\# V(\mathbb{F}_p) \approx p^{\dim V}.
\end{equation}

The Weil bound \eqref{weilbound}, and its close relatives, are very effective in problems for which $p$ grows or in problems for which $V$ varies through a family for which $b(V)$ is well-controlled.  However, if $p$ is fixed and $b(V)$ is large, the null hypothesis from \eqref{naivepointcountguess} seems a priori more suspect.

\begin{exam} \label{Agexample}
Let $A_g$ be the moduli space\footnote{$A_g$ is a smooth Deligne-Mumford stack over $\mathbb{Z}.$  Since the discussion of the present section is informal, we elide this issue.}  of principally polarized abelian varieties.  The Euler characteristic of $A_g$ equals $\prod_{n = 1}^g \zeta(1 - 2n),$ for $\zeta$ the Riemann zeta function, which has absolute value roughly $\exp(g^2 \log g).$  Therefore, $b(A_g) \geq \exp(g^2 \log g)$ and for fixed $p,$ the Weil bound \eqref{weilbound} gives almost no useful information.   
\end{exam}

\begin{exam} \label{Mgexample}
A closely related example, alluded to in the introduction is $M_g / \mathbb{F}_p,$ the moduli space of genus $g$ curves over $\mathbb{F}_p.$  The Euler characteristic of $M_g$ equals $\frac{1}{2 - 2g} \zeta(1 - 2g),$ which has absolute value roughly $\exp(g \log g).$  Even if one believes the null hypothesis $\# M_g(\mathbb{F}_p) \approx p^{\dim M_g} = p^{3g-3},$ the Weil bound is worlds away from proving this.  
\end{exam}

We say that families of varieties $\{ V \}$ for which $b(V)$ grows much faster than $p^{\dim V}$ exhibit \emph{topological explosion}; these are precisely the families for which one cannot prove anything like $\log \left( \# V(\mathbb{F}_p) \right) \approx \dim V$ using the Weil bound.  This paper arose in our attempts to overcome the topological explostion of the family of varieties $\{ A_g \}.$  More precisely, we hoped prove that there is \emph{significant cancellation} in the sums $S_i(A_g)$ for $i < 2\dim A_g$ by directly estimating $\#A_g(\mathbb{F}_p)$ using linear algebraic parametrizations of principally polarized abelian varieties.  Combined with the (to us, plausible) hypothesis
\begin{equation} \label{smallbettinumbers}
\log b(A_g) \ll g^2 \log g,
\end{equation} 
our work would prove that there is little cancellation in the sums $S_i(A_g).$  However, we are presently unable to prove \eqref{smallbettinumbers} or otherwise prove a lack of cancellation in the sums $S_i(A_g).$ Unfortuantely, we also have no idea how to address \ref{Mgexample}.  

\begin{rem}
Random-matrix models, predicting that $\frob_p$ acting on $H^i_{et}(V),$ for $V / \mathbb{F}_p$ varying through a family with large geometric monodromy, should behave like a random conjugacy class in the associated monodromy group (a compact, classical Lie group), give a very compelling explanation for the staggering cancellation in $S_i(V)$ which would be necessary for \eqref{naivepointcountguess} to hold.  See \cite{KatzSarnak} and \cite{Kedlayaetal} for further discussion.  
\end{rem}

\section{Bounding the number of isogeny classes} 
\label{isogeny}

\begin{lem} \label{numberofweilpolynomials}
The number of degree $2g$ monic integer polynomials $p(x)$ all of whose roots have absolute value $\sqrt{q}$ is at most $ (2g)^g q^{\frac{1}{4} g(g+1)}.$
\end{lem} 

\begin{proof}
Let $f(x) = x^{2g} - a_1 x^{2g-1} + \cdots - a_{2g-1} x + q^g.$ There is an equality 
$$q^g f(x) = x^{2g} f(q/x)$$ 
because both polynomials have the same roots leading coefficient.  
Therefore, $f(x)$ is uniquely determined by $a_1,\ldots, a_g,$ i.e. the first $g$ elementary symmetric functions of the roots of $f(x).$  These in turn are uniquely determined by $s_1,\ldots, s_g,$ where $s_g$ is sum of $k$th powers of the roots of $f(x).$  Every $s_k$ is an integer lying in the interval $[-g \sqrt{q}^k, g \sqrt{q}^k ].$  The number of possible such choice is at most
\begin{align*}
\prod_{k = 1}^g \left( 2g \sqrt{q}^k \right) &= (2g)^g q^{\frac{1}{2}\left(1 + \cdots + g \right)} \\
&= (2g)^g q^{\frac{1}{4} g(g+1)}.
\end{align*}    
\end{proof}  

\begin{cor} \label{upperboundisogenyclasses}
There are at most $(2g)^g q^{\frac{1}{4} g(g+1)}$ isogeny classes of abelian varieties over $\mathbb{F}_q.$
\end{cor}

\begin{proof}
By Tate's theorem, every isogeny class of $g$-dimensional abealian varieties over $\mathbb{F}_q$ is uniquely determined by its (geometric) Frobenius characteristic polynomial, an integer monic polynomial of degree $2g.$  By the Weil conjectures, the roots of geometric Frobenius are algebraic integers all of whose conjugates have absolute value $\sqrt{q}.$  The result thus follows immediately from Lemma \ref{numberofweilpolynomials}.
\end{proof}

\begin{rem}
Define the map
\begin{align*}
(S^1)^g &\rightarrow \mathbb{R}^g \\
(\alpha_1,\ldots,\alpha_g) &\mapsto (a_1,\ldots,a_g) 
\end{align*}
where  $(x - \sqrt{q}\alpha_1)(x - \sqrt{q} \overline{\alpha_1}) \cdots (x - \sqrt{q} \alpha_g)(x - \sqrt{q} \overline{\alpha_g}) = x^{2g} - a_1 x^{2g - 1} + \cdots + (-1)^g a_g x^g + \cdots$  The image of this map is a compact semialgebraic domain in $\mathbb{R}^g.$   DiPippo and Howe \cite{DH} prove a lower bound on the number of polynomials satisfying the hypotheses of Lemma \ref{numberofweilpolynomials} by proving that a reasonably large rectangle $R,$ centered at $(a_1,\ldots,a_g) = (0,\ldots,0),$ is contained in the interior of this domain and counting lattice points inside $R.$  The logarithm of their lower bound is asymptotic to $\frac{1}{4}g^2 \log q$ as $g$ grows, the same as the upper bound from Corollary \eqref{upperboundisogenyclasses}.

Over a prime field, this certifies that the logarithm of the upper bound from Corollary \ref{upperboundisogenyclasses} is asymptotically correct.  However, the characteristic polynomial of the Frobenius endomorphism of an abelian variety over $\mathbb{F}_{p^r},r > 1,$ must satsify a $p$-adic condition in addition to the hypotheses of Lemma \ref{numberofweilpolynomials}.  Namely, it must be the norm of a twisted conjugacy class in $GL_{2g}(\mathbb{Q}_{p^r})$ \cite{Clozel}.  We suspect that upper bound of $\frac{1}{4}g^2 \log q$ for the logarithm of the number of isogeny classes is nonetheless asymptotically correct, but we will not pursue this further.  
\end{rem}

\section{Bounding the number of isomorphism classes within a fixed isogeny class}
\label{isomorphism}

Let $A_0$ be a fixed abelian variety over $\FF_q.$  Let $\mathcal{O} = W(\FF_q)$ with fraction field $K.$ Let $D^0(A_0)$ denote the Dieudonne-module of $A_0$ and $D(A_0) = D^0(A_0) \otimes_{\mathcal{O}} K.$  There is a bijection between the set of 
$$A \xrightarrow{f} A_0,$$
where $A$ is an abelian variety over $\FF_q$ and $f$ is a quasi-isogeny and the set $X_p \times X^p$ where 
\begin{align*}
X^p &= \left\{ \prod_{\ell \neq p} L_{\ell}, L_{\ell}= \frob_q\text{-stable lattice inside } V_{\ell}(A_0), L_{\ell} = T_{\ell}(A_0) \text{ for almost all } \ell \right\} \\
&= \prod_{\ell \neq p} {}^{'} X^p_{\ell}
\end{align*}
$$X_p = \{  M: M \subset D(A_0) \text{ an } \langle F,V \rangle \text{-stable lattice} \}.$$

The bijection is given by 
$$\left(A \xrightarrow{f} A_0 \right) \mapsto f_{*} D^0(A) \times \prod_{\ell \neq p} f_{*} T_{\ell}(A).$$

To forget the quasi-isogeny, we mod out by the action of $\mathrm{End}^0(A_0)^{\times}$:  
\begin{equation} \label{isogenyclass}
\text{isogeny class of } A_0 \cong \mathrm{End}^0(A_0)^{\times} \backslash X_p \times X^p.
\end{equation}

Let $\G = \underline{\mathrm{End}^0(A_0)}^{\times}.$ 


\subsection{Orbits of $\G(\finadele)$ on $X_p \times X^p$}
Let $\gamma \in GL(V_{\ell}(A_0))$ denote the Frobenius element.  Let $Z_{\gamma}$ denote the centralizer of $\gamma.$  Let $\chi_{\gamma}(x) = f_1(x)^{n_1}\cdots f_j(x)^{n_j}$ be the characteristic polynomial of $\gamma,$ where every $f_i(x) \in \mathbb{\mathbb{Z}}_{\ell}[x]$ is irreducible.  Because $\gamma$ acts semisimply on $V_{\ell}(A_0),$ its minimal polynomial equals $f_1(x) \cdots f_j(x).$  There is a decomposition 
$$V_{\ell} = V_1 \oplus \cdots \oplus V_j, \text{ where } V_i = \ker f_i(\gamma).$$

\subsubsection{$\gamma$-stable lattices with a fixed isotypic decomposition} \label{directsum}
Fix $\gamma$-stable lattices $M_1 \subset V_1, \ldots , M_j \subset V_j.$  The number of lattices $M \subset V_{\ell}$ for which $M_i = \ker f_i(\gamma)|_M$ exactly equals $\ell^{\delta'},$ where $\delta' = \sum_{k \neq k'} \mathrm{val}_\ell \left(\mathrm{Res}(f_k, f_{k'}) \right)$ \cite[Proposition 4.9]{Yun}.  It therefore suffices to bound the number of $\gamma_i = \gamma|_{V_i}$ stable lattices in $V_i.$ 
 
\subsubsection{Finding orbit representatives with small colength in a fixed lattice} \label{smallcolength}
Let $R_i = \mathbb{Z}_{\ell}[x]/(f_i(x))$ and $F_i = \mathrm{Frac}(R_i).$ The $F = F_i$-vector space $V = V_i$ has dimension $n = n_i.$  We need to bound the number of orbits of $GL_n(F)$ on the space of $R$-lattices in $V \cong F^{\oplus n}.$  \medskip

There is a natural map from $R$-lattices in $V$ to $O_F$-lattices in $V$:
$$f: M \mapsto M O_F,$$ 
the $O_F$-lattice generated by $M.$  This map is $GL_n(F)$-equivariant.  Also, $GL_n(F)$ acts transitively on the collection of $O_F$-lattices and $f(R^n) = O_F^n.$  Therefore, it suffices to bound the orbits of $GL_n(O_F)$ acting on the fiber $f^{-1}(O_F^n).$ \medskip

Given $M \in f^{-1}(O_F^{\oplus n}).$  Because $MO_F = O_F^n,$ we can solve the system of equations
$m_1 = Ae_1, ..., m_n = Ae_n$ for the standard basis $e_1,...,e_n$ of $O_F^n,$ elements $m_1,...,m_n \in M$ which form an $O_F$-basis for $MO_F,$ and $A \in GL_n(O_F).$  Replacing $M$ by $M' = A^{-1} M,$ we get
$$A^{-1}m_1 = e_1,...,A^{-1}m_n = e_n. $$ 
Therefore, $M'$ contains $R^n$ but is contained in $O_F^n.$  The size of $O_F^n / R^n$ is at most the size of 
$(R^\vee / R)^n,$ which equals $\mathrm{disc}(R/\mathbb{Z}_{\ell})^n.$  Therefore, $M'$ is $GL_n(F)$-equivalent to a quotient of $R^{\oplus n}$ of length at most $\delta n,$ where $\delta = \mathrm{val}_{\ell}(\mathrm{disc}(R / \mathbb{Z}_\ell)).$

\subsubsection{A filtration argument for counting $\gamma$-stable lattices}
Fix a $\gamma$-stable filtration $V_1 \subset V.$  Fix $\gamma$-stable lattices $U_1 \subset V_1$ and $U_2 \subset V / V_1.$  Let $\pi: V \rightarrow V/V_1$ denote the projection.  We aim to count the number of $\gamma$-stable lattices $L \subset V$ satisfying $L \cap V_1 = U_1$ and $\pi(L) = U_2.$

Fix a $\gamma$-equivariant projection $p: V \rightarrow V_1.$  The pair $(L,p)$ gives rise to a $\mathbb{Z}_\ell[\gamma]$-linear 
map $\phi_L: U_2 \rightarrow V_1/U_1$ as follows:
$$\phi_L: u \rightarrow u' = \text{any lift of } u \text{ to } L \rightarrow p(u') \mod U_1.$$
Let $p': V \rightarrow V_1$ be a second $\gamma$-equivariant projection.  For any two lifts $u', u''$ of $u$ to $L,$ the difference $u' - u'' \in L \cap V_1 \subset V_1.$  Therefore, 
$$p(u'-u'') - p'(u'-u'') = (u' - u'') - (u' - u'') = 0.$$  
It follows that 
\begin{align*}
U_2 &\rightarrow V_1 \\
u &\mapsto (p - p')(\text{any lift of } u)
\end{align*} 
is a well-defined $\mathbb{Z}_\ell[\gamma]$-linear map.  Therefore, the assignment
\begin{align*}
\phi: \{ \text{lattices } L \subset V: L \cap V_1 = U_1 \text{ and } \pi(L) = U_2 \} &\rightarrow \Hom_{\mathbb{Z}_\ell[\gamma]}(U_2, V_1 / U_1) / \Hom_{\mathbb{Z}_\ell[\gamma]}(U_2,V_1) \\
L &\mapsto \phi_L
\end{align*}
is well-defined.   

\begin{lem}
The fibers of $\phi$ lie in the same $\G(\mathbb{Q}_{\ell})$-orbit of lattices.  In fact, they differ by a unipotent element of the centralizer.
\end{lem}

\begin{proof}
Suppose $\phi_{L_1} = \phi_{L_2}.$  Then there are $\mathbb{Z}_\ell[\gamma]$-linear projections $p_1: V \rightarrow V_1$ and $p_2: V \rightarrow V_1$ and a $\mathbb{Z}_\ell[\gamma]$-linear map $T: U_2 \rightarrow V_1$ satisfying
$$p_1(u') = p_2(u') + T(u)\mod U_1 \text{ for all } u \in L, u' = \pi(u) \in U_2.$$
There are $\mathbb{Q}_\ell[\gamma]$-stable decompositions 
\begin{align*}
V &= V_1 \oplus \ker(p_1) \\
V &= V_1 \oplus \ker(p_2)
\end{align*}
and $\mathbb{Q}_\ell[\gamma]$-isomorphisms
\begin{align*}
\ker(p_1) &\xrightarrow{i_1 = \pi} V/V_1, \\
\ker(p_2) &\xrightarrow{i_2 = \pi} V/V_1.
\end{align*}

The $\mathbb{Q}_\ell$-linear transformation 
\begin{align*}
V = V_1 \oplus \ker(p_2) &\rightarrow V = V_1 \oplus \ker(p_1) \\
a \oplus 0 &\mapsto a \oplus 0 \\
0 \oplus b  &\mapsto T \circ i_2(b) \oplus i_1^{-1} \circ i_2(b) 
\end{align*}
commutes with $\gamma$ and sends $L_2$ to $L_1.$
\end{proof}

\subsubsection{Bounding the size of  $\Hom_{\mathbb{Z}_\ell[\gamma]}(U_2, V_1 / U_1) / \Hom_{\mathbb{Z}_\ell[\gamma]}(U_2,V_1)$} \label{twostepfiltration}
Consider the diagram of $\mathbb{Z}_\ell$ modules
$$\begin{CD}
0 @>>> U_2^\ast \otimes_{\mathbb{Z}_\ell} U_1 @>>> U_2^\ast \otimes_{\mathbb{Z}_\ell} V_1 @>>> U_2^\ast \otimes_{\mathbb{Z}_\ell} V_1 / U_1 @>>> 0 \\
@. @VV{F_1}V @VV{F_2}V @VV{F_3}V @.\\
0 @>>> U_2^\ast \otimes_{\mathbb{Z}_\ell} U_1 @>>> U_2^\ast \otimes_{\mathbb{Z}_\ell} V_1 @>>> U_2^\ast \otimes_{\mathbb{Z}_\ell} V_1 / U_1 @>>> 0.
\end{CD}$$
The top and bottom rows are exact and $F = \gamma_2^{*} \otimes 1 - 1 \otimes \gamma_1.$  The Snake Lemma gives an isomorphism 
\begin{equation} \label{snakelemma}
\left( U_2^\ast \otimes_{\mathbb{Z}_\ell} V_1 / U_1 \right) [F_3] / \left( U_2^\ast \otimes_{\mathbb{Z}_\ell} V_1 \right) [F_2] \xrightarrow{\cong} \ker \left( \mathrm{coker}(F_1) \rightarrow \mathrm{coker}(F_2) \right).
\end{equation}

The left side is precisely what we're trying to bound.  The torus $\mathbf{T} = \underline{\mathbb{Q}_\ell[\gamma]}^\times$ acts semisimply on $U_2^\ast \otimes_{\mathbb{Z}_\ell} V_1 = V_2^\ast \otimes_{\mathbb{Q}_\ell} V_1.$  

Its invariant subspace $W$ equals the subspace acted on as 0 by its Lie algebra $ \mathbb{Q}_\ell[\gamma]$ with its induced action; this subspace equals $\ker(\gamma_2^\ast \otimes 1 - 1 \otimes \gamma_1) = \ker(F_2).$  Let $W'$ denote a $\mathbf{T}$-invariant subspace of $V_2^\ast \otimes_{\mathbb{Q}_\ell} V_1$ complemenetary to $W.$  

Since $F_2$ acts invertibly on $W'$ and kills $W,$ we identify $\mathrm{coker}(F_2) = W.$  Call $M = U_2^\ast \otimes_{\mathbb{Z}_\ell} U_1$ and $M' = M \cap W'.$  Then
\begin{align*}
\ker \left( \mathrm{coker}(F_1) \rightarrow \mathrm{coker}(F_2) \right) &= M' / ( F_1(M) \cap W') \\
&\subset M' / F_1(M').
\end{align*}
Therefore, the left side of \eqref{snakelemma} has size at most
$$\ell^v, \text{ where } v = \sum \mathrm{val}_{\ell}(\lambda_1 - \lambda_2),$$
where $\lambda_1, \lambda_2$ run over all pairs of \emph{unequal} eigenvalues of $\gamma_1$ acting on $V_1$ and $\gamma_2$ acting on $V_2$ respectively.

\subsubsection{Number of isotypic $\gamma$-stable lattices of bounded colength}
Reprise the notation from \S \ref{smallcolength}.  Let $L$ be an $R = \mathbb{Z}_\ell[x] / (f(x))$-lattice in $V \cong F^{\oplus n}.$  We may assume that $L$ is contained in $R^{\oplus n}$ and has colength at most $\delta n.$  Let 
$$V_{\leq i} = \mathrm{span}_F \langle e_1,\ldots, e_i \rangle, L_{\leq i} = L \cap V_{\leq i}, L_i = L_{\leq i} / L_{\leq i - 1} \text{ for all } i = 1,\ldots,n.$$

Let $a_i$ be the colength of $L_i$ in $R \overline{e_i}.$  Then $\sum_{i = 1}^n a_i \leq \delta n.$  Therefore,
\begin{align} \label{numberofRlattices}
&\# \{ R\text{-lattices } L \subset V: \mathrm{length}_{\mathbb{Z}_\ell} \left( R^{\oplus n} / L \right) \leq \delta n \} \nonumber \\
&\leq \sum_{a_1 + \cdots + a_n \leq \delta n} \sum_{\mathrm{length}(O_F \overline{e_i} / U_i) = a_i } \# \{ R\text{-lattices } L \subset V: L_i = U_i \} \nonumber \\
&\leq \ell^{n\delta}  \sum_{a_1 + \cdots + a_n \leq \delta n} \prod_{i=1}^n \# \{ R\text{-ideals of colength } a_i \} \nonumber \\
&=  \ell^{n\delta}  \sum_{a_1 + \cdots + a_n \leq \delta n} \prod_{i=1}^n \# \mathrm{Hilb}_R^{a_i}.
\end{align}

The second inequality in \eqref{numberofRlattices} follows from \S \ref{twostepfiltration} and an induction argument.  A keen insight of Yun \cite[\S 4.12]{Yun} is that $\# \mathrm{Hilb}_R^j \leq \# \mathrm{Hilb}_{\mathbb{Z}_\ell[[x]]}^j$ because $R$ is a quotient of the power series ring $\mathbb{Z}_\ell[[x]].$  Therefore, 
\begin{align} \label{hilbertschemepowerseries}
\# \mathrm{Hilb}_R^j &\leq \# \mathrm{Hilb}_{\mathbb{Z}_\ell[[x]]}^j \nonumber \\
&= \sum_{\lambda \text{ partition of } j} \ell^{j - \mathrm{length}(\lambda)} \nonumber \\
&\leq \ell^j \# \{ \text{partitions of } j \} \nonumber \\
&\leq \ell^j 2^j,
\end{align}
where the second line follows from \cite[Proposition 4.13]{Yun}.  Using \eqref{hilbertschemepowerseries}, we may continue our estimate in \eqref{numberofRlattices} by
\begin{align} \label{finalisotypiclocalestimate}
 \# \{ R\text{-lattices } L \subset V: \mathrm{length}_{\mathbb{Z}_\ell} \left( R^{\oplus n} / L \right) \leq \delta n \} &\leq \ell^{n \delta} \sum_{a_1 + \cdots + a_n \leq \delta n} \prod_{i = 1}^n \ell^{a_i} 2^{a_i} \nonumber \\
& \leq \ell^{n \delta} \ell^{n \delta} 2^{n \delta} \sum_{a_1 + \cdots + a_n \leq \delta n} 1 \nonumber \\
& \leq \ell^{2 n \delta} 2^{n \delta} (n \delta)^n / n! \nonumber \\
&\leq \ell^{2 n \delta} 2^{n \delta} e^n \delta^n \nonumber \\
&\leq \ell^{2 n \delta} 2^{2 n \delta} \nonumber \\
&\leq \ell^{4n \delta}.
\end{align}

\subsubsection{Number of $\gamma$-stable lattices of bounded colength} \label{numberoforbits}
Let $\delta_{k,k'} = \mathrm{val}_\ell(f_k,f_{k'}).$  Combining \S \ref{directsum} with the estimate \eqref{finalisotypiclocalestimate} gives
\begin{align*}
\# \{ \gamma\text{-stable lattices } L \subset V_{\ell}(A_0) \} / \G(\mathbb{Q}_\ell) &\leq \ell^{\sum_{i = 1}^j 4n_i \delta_i} \cdot \ell^{\sum_{k \neq k'} n_k n_k' \delta_{k,k'}} \\
&\leq \ell^{4 \sum_{\lambda \neq \mu} \mathrm{val}_{\ell}(\lambda - \mu)},
\end{align*}
where $\lambda, \mu$ run over all pairs of unequal roots of the characteristic polynomial $\chi_{\gamma}$ of $\gamma.$  Taking the product over all primes $\ell$ gives
\begin{align} \label{weilnumberestimate}
\# \{ \gamma\text{-stable lattices } L \subset V_{\mathrm{fin}}(A_0) \} / \G(\finadele) &\leq \left( \prod_{\lambda \neq \mu} | \lambda - \mu |_{\infty}  \right)^4 \nonumber \\
&\leq (2p^{\frac{1}{2}})^{4 \cdot {2g \choose 2}},
\end{align}

where the final estimate \eqref{weilnumberestimate} follows because all roots of $\chi_{\gamma}$ are $p$-Weil numbers.

\begin{rem}
In all local estimates applied so far, we've systematically ignored the prime $p.$  For abelian varieties over the prime field $\mathbb{F}_p,$ this does not pose any problem.  The simultaneous centralizer $\mathbf{Z}_{F,V}$ of the $\mathbb{Q}_p$-linear transformations $F = \gamma$ and $V = p \gamma^{-1}$ equals $\G,$ the centralizer of $\gamma.$  The number of $\mathbf{Z}_{F,V}$-orbits of lattices in the $\mathbb{Q}_p$-vector space $D(A_0)$ stable under $F$ and $V$ is bounded above by the number of $\G$-orbits on the collection of $F$-stable lattices.

Over more general finite fields $\mathbb{F}_{p^r}, D(A_0)$ is a $\mathbb{Q}_{p^r}$-vector space and the operators $F,V$ are only semilinear.  So more needs to be said in this case.
\end{rem}

\subsection{Number of isomorphism classes in a fixed isogeny class}
In \S \ref{numberoforbits}, we showed that the number of orbits of $\G(\finadele)$ on the space of $\widehat{\mathbb{Z}}$-lattices in $V_{\mathrm{fin}}(A_0)$ is at most $ (2p^{\frac{1}{2}})^{4 \cdot {2g \choose 2}}.$  Furthermore, every orbit contains a representative $L = \prod_v L_v$ for which $L_v \subset \bigoplus_i R_i^{\oplus n_i}$ of colength at most the $v$-adic valuation of the ``discriminant" of $\chi_\gamma(x),$ i.e. product of $\lambda - \mu$ over all pairs of unequal roots of $\chi_{\gamma}(x)$ (\S \ref{smallcolength}).  Therefore,
\begin{align} \label{upperboundisomorphismclasses}
&\# \{ \text{isomorphism classes in the isogeny class of } A_0 \} \nonumber \\
&\leq  (2p^{\frac{1}{2}})^{4 \cdot {2g \choose 2}} \max_{\text{orbit representatives } L} \{ \# \G(\mathbb{Q}) \backslash \G(\finadele) / \mathrm{Stab}_{\G(\finadele)}(L) \}. 
\end{align}

\subsubsection{Bounding the depth of $\mathrm{Stab}_{\G(\mathbb{Q}_v)}(L_v)$} \label{boundingdepth}
Let $R = \mathbb{Z}_\ell[x] / (f(x))$ for an irreducible polynomial $f(x) \in \mathbb{Z}_\ell[x].$  Let $F = \mathrm{Frac}(R).$  Suppose the characteristic polynomial of $\gamma$ acting on $V_\ell$ equals $\chi_\gamma = f(x)^n.$  Choose an $F$-basis so that $V_\ell \cong F^{\oplus n}.$  Adjusting $M = L_v$ by $\G(\mathbb{Q}_v) = GL_n(F),$ we may assume that $R^{\oplus n} \subset M \subset O_F^{\oplus n} \subset (R^\vee)^{\oplus n}.$  We want to bound the number of possibilities for $gM$ for $g \in GL_n(O_F).$  
\begin{itemize}
\item
Suppose that $g \in GL_n(R).$  There is an injective map
\begin{align*}
GL_n(R)\text{-orbit of } M &\rightarrow \Hom_R(M, \left( R^\vee / R \right)^{\oplus n} ) \\
gM &\mapsto \left( m \mapsto g(m) \mod R^n \right).
\end{align*}

So it suffices to bound $\# \Hom_R(M, \left( R^\vee / R \right)^{\oplus n}) = \# \left[ \Hom_R(M, R^\vee / R) \right]^n.$  Filter $M$ by its intersections $M_{\leq i} = M \cap \mathrm{span}_F \langle e_1,\ldots,e_i \rangle$ and let $M_i = M_{\leq i} / M_{\leq i-1}.$  Then
\begin{equation*}
\# \Hom_R(M, R^\vee / R ) \leq \prod_{i = 1}^n \#  \Hom_R(M_i, R^\vee / R). 
\end{equation*}

Every $M_i$ can be identified with an $R$-ideal $J$ of colength $a_i$ with $\sum a_i \leq \delta$ and $R^\vee / R \cong R / I$ where $I = f'(x)R.$  Therefore, we are reduced to bounding $\# \Hom_R(J,R/I).$ 

We have an exact sequence of $\mathbb{Z}_\ell$-modules
$$0 \rightarrow  I \rightarrow R \rightarrow R/I \rightarrow 0.$$ 

We have a commutative diagram  
$$\begin{CD}
0 @>>> J^\ast \otimes_{\mathbb{Z}_\ell} I =: A @>>> J^\ast \otimes_{\mathbb{Z}_\ell} R =: B @>>> J^\ast \otimes_{\mathbb{Z}_\ell} R/I @>>> 0 \\
@. @VV{F_1}V @VV{F_2}V @VV{F_3}V @. \\
0 @>>> J^\ast \otimes_{\mathbb{Z}_\ell} I =: A @>>> J^\ast \otimes_{\mathbb{Z}_\ell} R =: B @>>> J^\ast \otimes_{\mathbb{Z}_\ell} R/I @>>> 0,
\end{CD}$$
where $F_i = \gamma_2^\ast - 1 \otimes \gamma_1.$  Since $J$ is $\mathbb{Z}_{\ell}$-free, both rows are exact.  

By the Snake Lemma, there is an exact sequence
$$A[F_1] \rightarrow B[F_2] \rightarrow \Hom_R(J,R/I) \rightarrow \mathrm{coker}(F_1) \rightarrow \mathrm{coker}(F_2)$$
and in particular
\begin{equation} \label{boundingextensions}
\# \Hom_R(J,R/I) = \#  \mathrm{coker} \left( A[F_1] \rightarrow B[F_2] \right) \cdot \# \ker( \mathrm{coker}(F_1)  \rightarrow \mathrm{coker}(F_2) ).
\end{equation}

To bound the kernel in \eqref{boundingextensions}:  Because $F$ acts semisimply on $A_{\mathbb{Q}_\ell},$ we can decompose $A_{\mathbb{Q}_\ell} = V_0 \oplus V',$ where $V_0 = \ker F$ and $V'$ is an $F$-invariant complement.  Let $A'$ be the projection of $A$ to $V'.$  So $0 \rightarrow A[F] \rightarrow A \rightarrow A' \rightarrow 0$ is exact. Note that $A/FA$ surjects onto $A'/FA'$ with kernel most $A[F].$ Now $A'/FA'$ can be bounded using products of differences of eigenvalues as before.  Also, the kernel of $A[F] \rightarrow B/FB$ is trivial, since $F$ is semisimple (so in particuar, nothing in $A[F]$ can be in the image of $F$).  By the Snake Lemma exact sequence
$$0 \rightarrow \ker(A[F] \rightarrow B/FB) = 0 \rightarrow \ker(A/FA \rightarrow B/FB) \rightarrow \ker(A'/FA' \rightarrow 0) = A'/FA',$$
it follows that
\begin{equation*}
\# \ker(A/FA \rightarrow B/FB) \leq \# A'/FA' \leq \ell^\delta \text{ where } \delta = \mathrm{val}_\ell( \mathrm{disc}(R / \mathbb{Z}_\ell)).
\end{equation*}

To bound the cokernel in \eqref{boundingextensions}: Note that 
$$\Hom_R(R,I) \subset A[F] \subset B[F] \subset \Hom_R(I,R).$$

Therefore,
\begin{equation*}
\# \left( B[F] / A[F] \right) \leq \# \left( \Hom_R(I,R) / \Hom_R(R,I) \right) \leq \ell^{2\delta}.
\end{equation*}

Combining everything, it follows that
\begin{equation} \label{GLnRorbit}
\# \{ GL_n(R) \text{-orbit of } M \} \leq \ell^{3\delta n^2}.
\end{equation}

\item
The size of $GL_n(O_F) \backslash GL_n(R)$ is not too big.  

Let $\mu$ be an \emph{additive} Haar measure on $\End_n(O_F).$  For every $g \in GL_n(O_F)$ and every open $E \subset \End_n(O_F),$ 
\begin{equation} \label{multiplyadditivemeasure}
\mu(gE) = \mu(E).
\end{equation}  
Indeed, as a function of $E,$ the left side defines a Haar measure on $\End_n(O_F)$ and so equals a scalar multiple of the right side.  But both sides yield the same measure for $E = \End_n(O_F).$  Therefore,
$$\# \left( GL_n(O_F) / GL_n(R) \right) = \frac{\mu(GL_n(O_F))}{\mu(GL_n(R))}.$$
Let $m$ denote the size of the image of the reduction map $GL_n(R) \rightarrow GL_n(\mathbb{F}_\ell).$  By \eqref{multiplyadditivemeasure}, 
\begin{align*}
\frac{\mu(GL_n(O_F))}{\mu(GL_n(R))} &= \frac{\# GL_n(\mathbb{F}_\ell)}{ m } \cdot \frac{\mu(1 + \mathfrak{m}_{O_F} \End_n(O_F))}{\mu(1 + \mathfrak{m}_R \End_n(R)  )} \\
&=  \frac{\# GL_n(\mathbb{F}_\ell)}{ m } \cdot \frac{\mu(\mathfrak{m}_{O_F} \End_n(O_F))}{\mu( \mathfrak{m}_R \End_n(R)  )} \\
&=  \frac{\# GL_n(\mathbb{F}_\ell)}{ m } \cdot \frac{\frac{1}{\ell^{n^2}} \mu( \End_n(O_F))}{\frac{1}{\ell^{n^2}} \mu( \End_n(R)  )} \\
&= \frac{\# GL_n(\mathbb{F}_\ell)}{ m } \cdot \left( \# O_F / R \right)^{n^2} \\
&\leq \ell^{(1 + \delta)n^2}. 
\end{align*}
Together with \eqref{GLnRorbit}, this yields that 
\begin{equation} \label{GLnOForbit}
\# \{ GL_n(O_F)\text{-orbit of } M \} \leq \ell^{4\delta n^2}.
\end{equation} 
\end{itemize}

For more general, i.e. non-isotypic characteristic polynomials, our estimate \eqref{GLnOForbit} readily generalizes as in \S \ref{directsum} 
to give 
\begin{equation} \label{boundingdepthlatticestabilizer}
\# \mathbf{G}(\widehat{\mathbb{Z}}) / \mathrm{Stab}_{\G(\finadele)}(L) \leq \left(2p^{\frac{1}{2}}\right)^{4 {2g \choose 2}}.
\end{equation}

\subsubsection{Bounding the size of non-abelian class groups} \label{nonabelianclassgroupsize}
By \ref{endomorphismringprimefield}, the endomorphism ring of any simple abelian variety over the prime field $\mathbb{F}_p$ is \emph{commutative} unless the field generated by Frobenius is $\mathbb{Q}(\sqrt{p}).$  Let $B_0$ be a member of the unique isogeny class of (simple) abelian varieties over $\mathbb{F}_p$ whose Frobenius field is $\mathbb{Q}(\sqrt{p}).$  Let $D_0 = \End^0(B_0)$; note that $D_0$ is split at every finite place (see Proposition \ref{endomorphismringprimefield}).  Let $U_{0,d} \subset \GL_d(D)(\finadele)$ be the image of $GL_{4d}(\widehat{\mathbb{Z}})$ under a fixed isomorphism witnessing this splitness.  Let $\mathrm{Cl}(U_{0,1}) = \GL_1(D) (\mathbb{Q}) \backslash \GL_1(D)(\finadele) / U_{0,1}.$  Suppose the abelian variety $A$ is isogenous to $B_0^d \times A_1^{n_1} \times \cdots \times A_m^{n_m}$ where $A_i$ is simple with Frobenius field $K_i.$  Let $\G_i = \mathrm{Res}_{K_i / \mathbb{Q}} \GL_{n_i}.$  The units of the endomorphism algebra of $A$ equal
$$\G = \GL_d(D) \times \prod_{i = 1}^m \G_i.$$ 
The kernel of the reduced norm map (the determinant map) on every $\G_i$ and the kernel of the reduced norm map on $\GL_d(D), d \neq 1,$ is simply connected and has non-compact archimedean component.  By strong approximation, the product of reduced norm maps in every factor  
\begin{equation} \label{strongapproximation}
\G(\mathbb{Q}) \backslash \G(\finadele) / U_{0,d} \times \prod_{i = 1}^m \G_i(\widehat{\mathbb{Z}}) \rightarrow \begin{cases} \mathrm{Cl}(O_{\mathbb{Q}(\sqrt{p})}) \times \prod_{i = 1}^m \mathrm{Cl}(O_{K_i}) & \text{ if } d > 1  \\ \mathrm{Cl}(U_{0,1}) \times \prod_{i = 1}^m \mathrm{Cl}(O_{K_i}) & \text{ if } d = 1 \\ \prod_{i = 1}^m \mathrm{Cl}(O_{K_i}) & \text{ if } d = 0 \end{cases}
\end{equation} 
induces a bijection.  We need upper bounds for $\# \prod \mathrm{Cl}(O_K).$ 

Suppose $K/ \mathbb{Q}$ has degree $d = r_1 + 2r_2$ and discriminant $D_K.$  The class number formula tells us
\begin{equation} \label{classnumberformula}
\frac{2^{r_1} \cdot (2\pi)^{r_2} \cdot h_K R_k}{w_K \cdot \sqrt{D_K}} = \mathrm{Res}_{s = 1} \zeta_K(s).
\end{equation}  

There is a lower bound on the regulator \cite{Skoruppa}
\begin{equation} \label{regulatorlowerbound}
\frac{R_K}{w_K} \geq 0.00299 \cdot \exp(0.48 r_1 + 0.06 r_2) \geq \frac{1}{500}.
\end{equation}

There is an upper bound on the residue of the Dedekind zeta function \cite{Louboutin}
\begin{equation} \label{zetaresidueupperbound}
\mathrm{Res}_{s = 1} \zeta_K(s) \leq \left( \frac{e \log D_K}{2(d-1)} \right)^{d-1}.
\end{equation}

Taking the product of \eqref{classnumberformula} and the estimates \eqref{regulatorlowerbound} and \eqref{zetaresidueupperbound} over all fields $K$ appearing in the product $\prod \mathrm{Cl}(O_K)$ gives

\begin{align} \label{boundingclassnumber}
\# \prod \mathrm{Cl}(O_K) &= \prod h_K \nonumber \\
&=\prod \sqrt{D_K} \cdot  \prod \frac{w_K}{R_K} \cdot \frac{1}{2^{r_1} (2\pi)^{r_2}} \cdot \mathrm{Res}_{s=1} \zeta_K(s) \nonumber \\
&\leq \left( \prod \sqrt{D_K} \left( \log D_K \right)^{d-1} \right) \cdot 500^g e^g.
\end{align}

Remember that $O_K$ is the ring of intetegers of the endomorphism algebra of a simple isogeny factor of our abelian variety $A.$  Therefore, 
\begin{align} \label{smalldiscriminant}
|D_K| &\leq \mathrm{disc}( \text{some } p\text{-Weil number of degree } d) \nonumber \\
&\leq \left( 2 \sqrt{p} \right)^{{d \choose 2}}
\end{align}

Applying the estimate \eqref{smalldiscriminant} to \eqref{boundingclassnumber} gives
\begin{align} \label{upperboundclassnumber}
\# \prod \mathrm{Cl}(O_K) &\leq \left( \prod (2 \sqrt{p})^{\frac{1}{2}{d \choose 2}} \left({d \choose 2} \log \left(  2 \sqrt{p} \right)\right)^{d-1} \right) \cdot 500^g e^g \nonumber \\
&\leq \left( 2 \sqrt{p} \right)^{\frac{g^2}{2}} \cdot {g \choose 2}^g \cdot \left( \log \left( 2 \sqrt{p} \right) \right)^g 500^g e^g \nonumber \\
&= \left( 2 \sqrt{p} \right)^{\frac{g^2}{2}(1 + o(1))}.
\end{align}

\subsubsection{Conclusion}
We use the notation of \S \ref{nonabelianclassgroupsize}. Combining \eqref{upperboundisomorphismclasses}, \eqref{boundingdepthlatticestabilizer}, \eqref{strongapproximation}, and \eqref{upperboundclassnumber} yields
\begin{align} \label{finalupperboundisomorphismclasses}
&\# \{ \text{isomorphism classes in the isogeny class of } A_0 \} \nonumber \\
&\leq  (2p^{\frac{1}{2}})^{4 \cdot {2g \choose 2}} \max_{\text{orbit representatives } L} \{ \# \G(\mathbb{Q}) \backslash \G(\finadele) / \mathrm{Stab}_{\G(\finadele)}(L) \} \nonumber \\
&\leq (2p^{\frac{1}{2}})^{4 \cdot {2g \choose 2}}  \cdot \left(  \# \G(\mathbb{Q}) \backslash \G(\finadele) / \G(\widehat{\mathbb{Z}}) \right) \cdot  \max_{\text{orbit representatives } L} \# \left( \G(\widehat{\mathbb{Z}}) / \mathrm{Stab}_{\G(\finadele)}(L)  \right) \nonumber \\
&\leq  (2p^{\frac{1}{2}})^{4 \cdot {2g \choose 2}}  \cdot \left(  \# \G(\mathbb{Q}) \backslash \GL_d(D)(\finadele) \prod_{i = 1}^m \G_i(\finadele) / U_{0,d} \prod_{i = 1}^m \G_i(\widehat{\mathbb{Z}}) \right) \cdot \left(2p^{\frac{1}{2}}\right)^{4 {2g \choose 2}} \nonumber \\
&\leq C_0 \cdot  (2p^{\frac{1}{2}})^{4 \cdot {2g \choose 2}}  \cdot \left(  \# \prod \mathrm{Cl}(O_K) \right) \cdot \left(2p^{\frac{1}{2}}\right)^{4 {2g \choose 2}} \nonumber \\
&\leq  C_0 \cdot (2p^{\frac{1}{2}})^{4 \cdot {2g \choose 2}}  \cdot (2 p^{\frac{1}{2}})^{\frac{g^2}{2}(1 + o(1))} \cdot \left(2p^{\frac{1}{2}}\right)^{4 {2g \choose 2}} \nonumber \\
&= p^{\frac{33}{4}g^2 (1 + o(1))}.
\end{align}

In the above inequalities, the factor $C_0 := \max \{ \# \mathrm{Cl}(U_{0,1}), \# \mathrm{Cl}(O_{\mathbb{Q}(\sqrt{p})}) \}$ is only necessary if $d > 0,$ i.e. if $B_0$ occurs as an isogeny factor of $A.$  

\section{Polarizations} \label{polarizations}
\subsection{Preliminaries}
Let $A$ be an abelian variety over a field $k.$  Let $\mathcal{P}_A \rightarrow A \times A^\vee$ denote the universal line bundle.
\begin{defn}
A \emph{polarization} is a symmetric homomorphism $f: A \rightarrow A^\vee$ for which $(1,f)^\ast \mathcal{P}_A$ is ample.  A \emph{principal polarization} is a polarization which is an isomorphism.
\end{defn}

An important construction of symmetric homomorphisms $A \rightarrow A^\vee$, due to Mumford, runs as follows:
Let $\mathcal{L} \rightarrow A$ be a line bundle.  Then $m^\ast \mathcal{L} \otimes \pi_1^\ast( \mathcal{L})^{-1} \otimes \pi_2^\ast (\mathcal{L})^{-1} \rightarrow A \times A$ is a line bundle restricting trivially to $A \times 0$ and $0 \times A.$  There is therefore a homomorphism $\phi_\mathcal{L}: A \rightarrow A^\vee$ for which $(1,\phi_\mathcal{L})^\ast(\mathcal{P}_A) = \mathcal{L}.$  On points, 
$$\phi_{\mathcal{L}}(a) = t_a^\ast \mathcal{L} \otimes \mathcal{L}^{-1}.$$
Some important facts about Mumford's construction \cite[Theorem 5.6]{Conrad}:
\begin{itemize}
\item[(1)]
The homomorphism $\phi_\mathcal{L}$ determines $\mathcal{L}$ modulo tensoring with a translation invariant line bundle defined over $k,$ i.e. an element of $A^\vee(k).$  

\item[(2)]
Over a separably closed field, every symmetric homomorphism arises from the Mumford construction.  

\item[(3)]
If $\mathcal{L}$ is ample, then $\phi_\mathcal{L}$ is an isogeny.
\end{itemize}

\begin{rem}
By (2) and (1), the obstruction to expressing a symmetric homomorphism $f$ as $\phi_{\mathcal{L}}$ for some $\mathcal{L} \rightarrow A/k$ lies in $H^1(k, A^\vee).$  In particular, if $k$ is a finite field, the vanishing of $H^1(k,A^\vee)$ implies that $f = \phi_\mathcal{L}$ for some line bundle defined over $k.$ 
\end{rem} 

\begin{defn} \label{equivalentpolarizations}
An \emph{isomorphism of symmetric homomorphisms} $(f: A \rightarrow A^\vee) \rightarrow (g: B \rightarrow B^\vee)$ is an isomorphism $\alpha: A \rightarrow B$ for which
$$f = \alpha^\ast(g) := \alpha^\vee  g  \alpha: A \rightarrow A^\vee.$$
\end{defn}

\begin{rem}
The computation
\begin{align} \label{pullbackpolarization}
\phi_{\alpha^\ast \mathcal{L}}(x) &= t_x^\ast (\alpha^\ast \mathcal{L}) \otimes (\alpha^\ast \mathcal{L})^{-1} \nonumber \\
&= \alpha^\ast(t_{\alpha(x)}^\ast \mathcal{L}) \otimes \alpha^\ast (\mathcal{L})^{-1} \nonumber \\
&= \alpha^\ast ( t_{\alpha(x)}^{\ast} \mathcal{L} \otimes \mathcal{L}^{-1} ) \nonumber \\
&= \alpha^\vee (\phi_\mathcal{L} (\alpha(x)))
\end{align}
shows that equivalence via Definition \ref{equivalentpolarizations} is compatible with line bundle preserving isomorphisms through the Mumford construction.
\end{rem}

\subsection{The Rosati involution, the Neron-Severi lattice, and symmetric elements of the endomorphism algebra}
Let $f = \phi_{\mathcal{L}_0}$ be a principal polarization of $A.$  

\begin{defn}
The \emph{Rosati involution associated to $f$} is defined by 
\begin{align*}
': \End^0(A) &\rightarrow \End^0(A) \\
h &\mapsto h' := f^{-1} \circ h^\vee \circ f. 
\end{align*}
Define $S^0(A) \subset \End^0(A)$ to be the subspace fixed by $'.$
\end{defn}

The polarization $f$ allows us to define a map 
\begin{align*}
\Phi_f: \NS(A) = \Pic(A) / \Pic^0(A) &\rightarrow S^0(A) \\
\mathcal{L} &\mapsto \frac{1}{2} f^{-1} \phi_{\mathcal{L}}.
\end{align*}

Pullback of polarization corresponds through $\Phi_f$ to Rosati conjugation by \eqref{pullbackpolarization}:
\begin{align*}
\Phi_f(\alpha^\ast \mathcal{L}) &= \frac{1}{2} f^{-1} \alpha^\vee \phi_\mathcal{L} \alpha \\
&= f^{-1} \alpha^\vee f \left(\frac{1}{2} f^{-1} \phi_{\mathcal{L}} \right) \alpha \\
&= \alpha' \Phi_f(\mathcal{L}) \alpha.
\end{align*}

Define a bilinear form 
\begin{align*}
D_f: \End^0(A) \times \End^0(A) &\rightarrow \NS(A) \\
(a,b) &\mapsto (a + b)^\ast \mathcal{L}_0 \otimes  a^\ast(\mathcal{L}_0)^{-1} \otimes b^\ast(\mathcal{L}_0)^{-1}.
\end{align*}

To see bilinearity, use the injectivity of $\mathcal{L} \mapsto \phi_{\mathcal{L}}$ and \eqref{pullbackpolarization}: 
\begin{align} \label{bilinearity}
\phi_{D_f(a,b)} &= (a + b)^\vee  f  (a+b) - a^\vee  f  a - b^\vee  f b \nonumber \\
&= a^\vee f  b + b^\vee f  a
\end{align}
since $(a + b)^\vee = a^\vee + b^\vee.$

\begin{lem}
$\Phi_f$ and $D_f: a \mapsto D_f(a,1)$ are inverse isomorphisms.
\end{lem}

\begin{proof}
Use \eqref{bilinearity}.  See \cite[\S 4]{Lang}.
\end{proof}

\subsection{The trace, positive endomorphisms, and the ample cone}
Let $A/k$ be a $g$-dimensional abelian variety.  For any endomorphism $h$ of $A$ and any prime $\ell \neq \mathrm{char}(k)$ let $P_h(x) = \det(x \cdot 1 - h | T_\ell(A)).$  This polynomial is independent of $\ell$ as can be seen by the fundamental equality \cite[\S 19, Theorem 4]{Mumford}
$$P_h(x) = \mathrm{deg}(x \cdot 1 - h) = x^{2g} - c_{2g-1}(h) x^{2g-1} + \cdots + c_0(h),$$
where $\mathrm{deg}$ is extended to $\End^0(A)$ by degree $2g$-homogeneity.

\begin{defn}
Define the \emph{trace} of $h \in \End^0(A)$ to be \cite[\S 2]{Lang}
\begin{align*}
t(h) &:= c_{2g-1}(h) \\
&= \frac{g}{\deg(\mathcal{L}_0^{(g)})} \deg \left(\mathcal{L}_0^{(g-1)} \cdot D_f(\alpha) \right).
\end{align*}
This gives rise to the scalar product on $\End^0(A)$ \cite[\S 2]{Lang} 
\begin{align*}
\langle a,b \rangle &:= t(a'b) \\ 
&= \frac{g}{\deg(\mathcal{L}_0^{(g)})} \deg \left(\mathcal{L}_0^{(g-1)} \cdot D_f(\alpha,\beta) \right).
\end{align*}
\end{defn}
By ampleness of $\mathcal{L}_0$ combined with $D_f(a,a) =  (a^\ast \mathcal{L}_0)^{\otimes 2},$ it follows immediately that $\langle \cdot, \cdot \rangle$ is positive definite.

\begin{rem}
If $A/k$ is simple with endomorphism algebra $D,$ then $t = \text{reduced trace}_{D/\mathbb{Q}}.$
\end{rem}

\subsubsection{Positivity and ampleness}
\begin{prop}[\cite{Lang},Theorem 3] \label{positivity}
The following conditions on $a \in S^0(A)$ are equivalent:
\begin{itemize}
\item[(1)]
$a = \Phi_f(\mathcal{L})$ for some ample line bundle $\mathcal{L}$ on $A.$

\item[(2)]
All roots of $P_a(x)$ are totally real and totally positive.

\item[(3)]
$a = $ sum of (invertible) squares in $\mathbb{Q}[a].$
\end{itemize}
\end{prop}

\begin{proof}
Key ideas: 
\begin{itemize}
\item[$(1) \implies (2)$]
Since $D_f(b'ab) = \beta^\ast D_f(a) = \beta^\ast \mathcal{L},$ 
$$t(b'ab) = \frac{g}{\deg(\mathcal{L}_0^{(g)})} \deg \left(\mathcal{L}_0^{(g-1)} \cdot \beta^\ast \mathcal{L} \right) \geq 0$$
by ampleness of $\mathcal{L}_0.$  Applying this for $b \in \mathbb{Q}[a]$ suffices to force the roots of $P_a(x)$ to be totally real and totally positive \cite[Theorem 2, part 3]{Lang}.

\item[$(3) \implies (1)$]
Suppose $a = \sum a_i^2.$  Then $\Phi_f(a_i^\ast \mathcal{L}_0) = a_i'a = a_i^2.$  Since $a_i$ is invertible, it is a \emph{finite} map.  Thus $a_i^\ast \mathcal{L}_0$ is ample for every $i$ and $\Phi_f(\bigotimes_i a_i^\ast \mathcal{L}^0) = a.$
\end{itemize}
See \cite{Lang} for further details.
\end{proof}

\subsubsection{Classifcation of possible ample cones}
Extend $t$ to $\End^0(A)_{\RR}.$  Let $a \in \End^0(A)_{\RR}$ be invertible and symmetric.  By the proof of Proposition \ref{positivity}, if $t(b'ab) = t(abb') = \langle a,bb' \rangle \geq 0$ for all $b \in \End^0(A)_{\RR}$ then $a$ is positive.  Therefore, the ample cone is self-dual.  

Furthermore, any invertible positive elment in $S^0(A)_{\RR},$ being a sum of squares, is actually of the form $b'b$ for some invertible $b \in \End^0(A)_{\RR}.$  Note that $b'b$ is the Rosati-conjugate of $1$ by $b \in \G(\mathbb{R}) = (\End^0(A)_{\RR})^\times.$  It follows that 
$$(\text{ample cone}, \langle \cdot, \cdot \rangle) \subset (S^0(A)_{\RR}, \langle \cdot, \cdot \rangle)$$ 
with the Rosati conjugation $\G(\RR)$-action is a self-dual homogeneous cone.

Remarkably, self-dual homogeneous cones have been completely classified and there are very few possibilities.  Even fewer arise as the ample cone of an abelian variety.

\begin{prop}[Koecher-Vinberg]
Suppose 
$$\End^0(A)_\RR \cong \bigoplus_i \End_{r_i}(\RR) \oplus \bigoplus_j \End_{s_j}(\CC) \oplus \bigoplus_k \End_{t_k}(\mathbb{H}).$$
As a $\G(\RR)$ self-dual cone, the ample cone of $A$ is isomorphic to 
$$\bigoplus_i \mathcal{P}_{r_i}(\RR) \oplus \bigoplus_j \mathcal{P}_{s_j}(\CC) \oplus \bigoplus_k \mathcal{P}_{t_k}(\mathbb{H}),$$
where for $F= \RR, \CC,$ or $\mathbb{H},\ast =$ conjugate transpose,
$$\mathcal{P}_m(F) = \text{positive definite } \ast\text{-symmetric endomorphisms of } F^{\oplus m}.$$
Here, the symmetric endomorphisms are given inner product $\langle a,b \rangle = \mathrm{tr}(a^\ast b)$ and $\G(\RR)$ acts by $a \mapsto g^\ast a g$ on every summand.
\end{prop}

\begin{proof}
This statement appears in \cite[Theorem 4.3]{ps}. See also \cite[\S 21, Applications 1 and 3]{Mumford}.
\end{proof}

\subsection{Counting principal polarizations on a fixed abelian variety} \label{countingpolarizations}
Summarizing all of the preceeding results on polarizations, as they pertain to polarization counting:  
\begin{prop} \label{orbitcountingprincipalpolarization}
Let $A/k$ be an abelian variety admitting a principal polarization $f.$  The principal polarizations of $A$ are in bijection with the orbits of $\mathrm{Aut}(A)$ acting on the automorphisms lying in the (symmetric) ample cone of $A$ by $f$-Rosati conjugation.  Concretely, the ample cone, as a $\G(\RR)$-space, is a direct sum of cones $\mathcal{P}_m(F), F = \RR, \CC, \text{ or } \mathbb{H}$ acted on by $\ast$-conjugation and the integral symmetric automorphisms of are ``determinant 1 lattice points" in $\bigoplus \mathcal{P}_m(F).$  ``Lattice point" refers to the integral structure induced by $\End(A).$
\end{prop}


\subsubsection{Two important examples}
\begin{exam} \label{highlyreducible}
Let $E/k$ be an elliptic curve; $E$ is canonically polarized.  Suppose $\End(E) = \mathbb{Z}.$  Let $A = E^g.$  By \S \ref{countingpolarizations}, isomorphism classes of principal polarizations on $A$ are in bijection with symmetric positive definite matrices in $GL_g(\mathbb{Z})$ modulo the action $g\cdot a = g^t a g$ of $GL_g(\mathbb{Z}).$  There are $\approx \exp(g^2 \log g)$ such orbits, cf. Lemma \ref{goodellipticcurves}.
\end{exam}

\begin{exam} \label{irreducible}
Let $A/k$ be an abelian variety with $\End(A) \cong O_L$ for some CM-field $L.$  The Rosati involution is given by complex conjugation.  Let $K \subset L$ be the fixed field of complex conjugation.  By \S \ref{countingpolarizations}, isomorphism classes of polarizations of $A$ are in bijection with $\frac{(O_K^{\times})^{+}}{\mathrm{Norm}_{L/K}(O_L^\times)},$ where $(O_K^\times)^{+}$ denotes the totally positive units of $O_K.$  The localization homomorphism
$$O_K^\times \rightarrow \prod_v \frac{O_{K,v}^{\times}}{N(O_{L,v}^{\times})}$$
has kernel those units which are everywhere locally norms.  By Hasse's norm theorem applied to the cyclic Galois extension $L/K,$ such units lie in $N(L^{\times}) \cap O_K^{\times}.$  The quotient $\frac{ N(L^{\times}) \cap O_K^{\times}}{ N(O_L^\times)}$ is naturally a subquotient of the class group $\mathrm{Cl}(O_L)$ \cite{Lemm}.  Therefore,
\begin{align} \label{boundingpolarizationscommutative}
\# \frac{O_K^{\times}}{ N(O_L^\times)} &\leq \# \mathrm{Cl}(O_L) \cdot \# \prod_v \frac{O_{K,v}^{\times}}{N(O_{L,v}^{\times})} \nonumber \\
&\leq (2p^{\frac{1}{2}})^{\frac{g^2}{2}(1 + o(1))} \cdot  (2p^{\frac{1}{2}})^{\frac{g^2}{2}(1 + o(1))},
\end{align}
where the second estimate follows from \eqref{upperboundclassnumber}.  Therefore, the number of isomorphism classes of principal polarizations on $A$ is at most $\approx \exp(g^2).$
\end{exam}

Note that the number of polarizations in the first example is \emph{much greater} than the number of polarizations in the second example.  We expect, in general, that highly reducible (up to isogeny) abelian varieties admit \emph{many more} principal polarizations than nearly simple ones.

\subsection{Surprising consequences for abelian variety statistics}
The endomorphism ring of a simple abelian variety over the prime field $\mathbb{F}_p$ together with its Rosati involution admits a nice description.

\begin{prop}[\cite{Waterhouse}, Theorem 6.1] \label{endomorphismringprimefield}
Let $A / \mathbb{F}_p$ be a simple abelian variety.  Assume $\mathbb{Q}(\frob)$ contains no real prime.  Then:
\begin{itemize}
\item[(1)]
$E = \End^0(A)$ is commutative.

\item[(2)]
Any order $R$ in $E$ containing $\frob$ and $p \frob^{-1}$ is the endomorphism ring of some abelian variety $A' / \mathbb{F}_p.$  
\end{itemize}
\end{prop}

\begin{proof}
\begin{itemize}
\item[(1)]
We know a priori that $D = \End^0(A)$ is split at all finite places $v \nmid p$ of $\mathbb{Q}(\frob).$  Since $\mathbb{Q}(\frob)$ is totally complex, $D$ is split at all archimedean places too.  At places $v \mid p,$ Tate \cite{Tate} computed $\mathrm{inv}_v(D)$ and showed that the denominator divides $a,$ where $q = p^a.$  Therefore, $D_v$ is split for all $v \mid p$ if $a = 1.$   

Splitness at places $v \mid p$ can also be seen by Tate's theorem: $\End(A) \otimes \mathbb{Q}_p \cong \End_{F,V}(D^0(A)).$  When $a = 1,$ the right side equals the centralizer of $F,$ which is not a divison algebra.

\item[(2)]
This follows by the main construction of \cite{Waterhouse}.
\end{itemize} 
\end{proof}

\begin{rem}
Proposition \ref{endomorphismringprimefield} is false for $A$ not defined over the prime field $\mathbb{F}_p,$ e.g. supersingular elliptic curves over $\mathbb{F}_{p^2}.$  One key distinction between abelian varieties $A$ over prime fields and non-prime fields: the twisted centralizer of Frobenius acting on $D^0(A) / W(\mathbb{F}_q),$ which can be very complicated for $A$ over non-prime fields, reduces to the centralizer of a single element in $GL(D^0(A) / \mathbb{Q}_p)$    
\end{rem}

Albert classified all division algebras with positive involution.  For all cases occuring in Proposition \ref{endomorphismringprimefield}, the Rosati involution is necessarily complex conjugation \cite[\S 21]{Mumford}.

\begin{prop} \label{fewsquarefreepolarizations}
The number of principal polarizations on a $g$-dimensional abelian variety $A$ not containing any repeated simple isogeny factors and whose Frobenius characteristic polynomial is relatively prime to $x^2 - p$ is $\ll p^{C g^2}$ for some absolute constant $C.$
\end{prop}

\begin{proof}
By our first assumption, $A$ is isogenous to $A_1 \times \cdots \times A_n$ where every $A_i$ is simple and the Frobenius characteristic polynomials $\chi_i$ of $A_i$ are all irreducible over $\mathbb{Q}$ and distinct.  Since any $p$-Weil polynomial with a real root is necessarily divisible by $x^2 - p,$ our second assumption implies that no $K_i = \mathbb{Q}(\frob_i)$ contains a real prime.

Since the $A_i$ share no common isogeny factor, the endomorphism algebra decomposes
$$\End^0(A) = \prod \End^0(A_i).$$
Applying Proposition \ref{endomorphismringprimefield}, every $\End^0(A_i)$ is a \emph{commutative} division algebra, so
$$\End^0(A_i) = L_i,$$
where every $L_i$ is a CM-field.  By Albert's classification of division algebras with positive involution \cite[\S 21]{Mumford}, the Rosati involution on $\End^0(K_i)$ is necessarily given by complex conjugation on every simple factor.

Let $R = \End(A)$ and $K_i =$ subfield of $L_i$ fixed by complex conjugation.  As in Example \ref{irreducible}, if $A$ admits at least one principal polarization, 
\begin{equation} \label{polarizationset}
\{ \text{principal polarizations on } A  \} \cong \frac{R^{\times, +}}{N_{\End^0(A) / \prod K_i}(R^\times)},
\end{equation}
where $R^{\times,+}$ denotes those elements of $R^\times$ fixed by complex conjugation and totally positive in every simple factor.  Also, 
\begin{equation} \label{biggerpolarizationset}
 \frac{R^{\times, +}}{N_{\End^0(A) / \prod K_i}(R^\times)} \subset \frac{\prod O_{K_i}^{\times, +}}{N_{\prod L_i / \prod K_i}(R^\times)}.
\end{equation} 
By \eqref{boundingdepthlatticestabilizer}, the quotient $\frac{\prod O_{K_i}^{\times}}{ R^{\times}}$ has size at most $(2p^{\frac{1}{2}})^{4 {g \choose 2}}.$  Combining \eqref{polarizationset}, \eqref{biggerpolarizationset} with the upper bound from \eqref{boundingpolarizationscommutative}:
\begin{align} \label{squarefreeprincipalpolarization}
&\# \{ \text{principal polarizations on } A  \} \nonumber \\
&= \# \frac{R^{\times, +}}{N_{\End^0(A) / \prod K_i}(R^\times)} \nonumber\\
&\leq \# \frac{\prod O_{K_i}^{\times, +}}{N_{\prod L_i / \prod K_i}(R^\times)} \nonumber\\
&\leq \# \frac{\prod O_{K_i}^{\times, +}}{N_{\prod L_i / \prod K_i}(R^\times)} \nonumber\\
&\leq (2p^{\frac{1}{2}})^{4 {g \choose 2}} \cdot \# \prod \frac{O_{K_i}^{\times,+}}{N_{L_i/K_i}(O_{L_i}^{\times})} \nonumber\\
&\leq  (2p^{\frac{1}{2}})^{4 {g \choose 2}} \cdot \prod_{i=1}^n \left(  \# \mathrm{Cl}(O_{L_i}) \cdot \# \prod_v \frac{O_{K_{i,v}}^{\times}}{N(O_{L_{i,v}}^{\times})}\right) \nonumber\\
&\leq (2p^{\frac{1}{2}})^{4 {g \choose 2}} \cdot (2p^{\frac{1}{2}})^{g^2(1 + o(1))}.
\end{align}
The estimate \eqref{squarefreeprincipalpolarization} proves the proposition.
\end{proof}

Proposition \ref{fewsquarefreepolarizations} has some suprirsing consequences for the statistics of abelian varieties over prime fields.

\begin{prop} \label{lowprobabilitysquarefree}
Suppose the prime $p$ satisfies the conclusion of Lemma \ref{goodellipticcurves}.  The probability that a principally polarized abelian variety over the prime field $\mathbb{F}_p$ has no repeated isogeny factors and whose Frobenius characteristic polynomial is relatively prime to $x^2 - p$ approaches 0 as $g$ increases.
\end{prop}

\begin{proof}
By \eqref{finalupperboundisomorphismclasses}, there are $\leq (2p^{\frac{1}{2}})^{\frac{33}{4}g^2(1 + o(1))}$ isomorphism classes of abelian varieties in any individual isogeny class.  The upper bound from \S \ref{isogeny} says that the number of isogeny classes of abelian varieties over $\mathbb{F}_p$ is at most $p^{\frac{g^2}{4}(1 + o(1))}.$  Finally, Proposition \ref{fewsquarefreepolarizations} tells us that the number of principal polarizations on any fixed isomorphism class as in the proposition is $\leq (2p^{\frac{1}{2}})^{3g^2(1 + o(1))}.$  Therefore, the number of principally polarized, squarefree, $g$-dimensional abelian varieties over $\mathbb{F}_p$ is at most $p^{C g^2}$ for some absolute constant $C.$

This total number of polarizations is far less than the number of principal polarizations of $E^g$ for any fixed elliptic curve $E / \mathbb{F}_p.$  For example, let $E/\mathbb{F}_p$ with endomorphism ring $O_K,$ the full ring of integers of an imaginary quadratic field $K.$  The canonical principal polarization of $E$ yields a principal polarization on $E^g.$  The endomorphism ring and its corresponding Rosati involution are given by $(\End_g(O_K), \ast = \text{conjugate transpose}).$  As in Example \ref{highlyreducible}, principal polarizations on $E^g$ are in bijection with orbits of $GL_g(O_K)$ on the positive definite $\ast$-symmetric matrices in $GL_g(O_K).$  The number of such orbits grows like $\exp(g^2 \log g),$ cf. Lemma \ref{goodellipticcurves}.  Since the number of principal polarizations on $E^g$ is vastly greater than the number of principal polarizations on all squarefree abelian varieties ($\leq p^{C g^2}$) as $g$ grows, the probability that a principally polarized abelian variety is squarefree approaches 0 as $g$ grows. 
\end{proof}

\begin{rem}
We expect the conclusions of Propositions \ref{fewsquarefreepolarizations} and \ref{lowprobabilitysquarefree} will continue to hold for abelian varieties over non-prime finite fields.  But because endomorphism algebras of simple abelian varieties over non-prime fields can be \emph{non-commutative}, the proofs will necessarily be different.  In particular, the endomorphism algebra alone does not uniquely determine the Rosati involution. 
\end{rem}

\section{Speculation}
\subsection{Polarization counting and orbit counting on model rings with positive involution}
\begin{defn}
Let $\mathbf{L} = \{ L_i \}_{i=1}^k$ be a collection of CM-fields with totally real subfields $K_i,$ and $\sum_i n_i [L_i: \mathbb{Q}] = 2g$ for positive integers $\mathbf{n} = \{ n_i \}.$  Let $\overline{\bullet}$ denote the intrinsic complex conjugation of every $L_i.$ We call the pair
$$R_{\mathbf{L}, \mathbf{n}} = \prod \End_{n_i}(O_{L_i}), \ast: (M_1,\cdots,M_k)^\ast = (\overline{M_1}^{\mathrm{tr}},\cdots,\overline{M_k}^{\mathrm{tr}})$$
a \emph{model ring with positive involution} (MRPI).
\end{defn}

By Proposition \ref{orbitcountingprincipalpolarization}, if $(R_{\mathbf{L}, \mathbf{n}},\ast)$ is realized by a principally polarized abelian variety $A$ and its associated Rosati involution, the set of principal polarizations of $A$ is in bijection with orbits of $R_{\mathbf{L}, \mathbf{n}}^\times$ on the hermitian positive elements of $R_{\mathbf{L},\mathbf{n}}^\times$; the action is given by $g \cdot a = g^\ast a g.$

According to \cite{Waterhouse}, every abelian variety $A_0$ over a finite field with endomorphism algebra $\End^0(A) = \prod \End_{n_i}(L_i)$ is isogenous to another abelian variety $A$ with endomorphism ring $\End(A) = \prod \End_{n_i}(O_{L_i}),$ the ring underlying an MRPI.  It is not clear whether $A$ admits any principal polarization at all, or whether there exists a second $A'$ isogenous to $A_0$ with endomorphism ring $\End(A') = \End(A),$ which is principally polarized, and whose associated Rosati involution gives $\ast.$  Nonetheless, we believe it should be possible to relate principal polarization counts on $A_0$ to (non-principal) polarization counts on $A'.$  Even failing to realize $(R_{\mathbf{L}, \mathbf{n}}, \ast)$ by a principally polarized abelian variety, we believe that orbit counts on MRPIs will be comparable to principal polarization counts at a gross scale:

\begin{conj} \label{mrpiconjecture}
Let the $g$-dimensional abelian variety $A / \mathbb{F}_p$ be isogenous to $A_1^{n_1} \times \cdots \times A_1^{n_k}$ where every $A_i$ are distinct simple abelian varieties.  Suppose $\End^0(A_i) \cong L_i$ for CM-fields $L_i.$  Suppose $A$ admits at least one principal polarization, and let $n_A$ be the number of isomorphism classes of principal polarizations on $A.$  Let $n_{\mathbf{L},\mathbf{n}}$ be the number of orbits of $R_{\mathbf{L},\mathbf{n}}^\times$ on the symmetric positive elements of $R_{\mathbf{L},\mathbf{n}}^\times$ under $\ast$-conjugation.  Then
$$\log n_A - \log n_{\mathbf{L},\mathbf{n}} = O(g^2),$$
where the implicit constant depends only on $p.$ 
\end{conj}

Next, we'll describe how to estimate the quantity $\log n_{\mathbf{L},\mathbf{n}}$ from Conjecture \ref{mrpiconjecture} using the Siegel mass formula for positive definite hermitian lattices.  With these estimates in hand, \S \ref{orbitcountconsequences}, \S \ref{archimedeanrandommatrix}, and \S \ref{padicrandommatrix} describe some surprising consequences Conjecture \ref{mrpiconjecture} has for the statistics of abelian varieties over prime fields. \bigskip

Let $L$ be a CM-field with real subfield $K.$  In this section, we identify orbits of $GL_n(O_L)$ acting on $GL_n(O_L)^{+},$ the hermitian, positive definite elements of $GL_n(O_L),$ with a union of class groups of unitary groups over $K.$  

\subsection{$\ast$-conjugation orbits on $GL_n(O_L)^{+}$ and unimodular hermitian lattices} \label{astconjugationorbitsunimodularhermitianlattices}
Let $\langle z, w \rangle = \sum_{i=1}^n z_i \overline{w_i}$ be the standard hermitian form on $L^{\oplus n}.$  

For every $\sigma \in \Sigma = \frac{O_K^{\times,+}}{N(O_L^{\times})},$ fix a representative $\epsilon_{\sigma}.$  Let $\epsilon_1 = 1.$ Define $\langle z,w \rangle_{\sigma} := \epsilon_{\sigma} z_1 \overline{w_1} + \sum_{i = 2}^n z_i \overline{w_i}.$  In particular, $\langle \cdot, \cdot \rangle_1 = \langle \cdot,\cdot \rangle.$  Let $\mathbf{U}_{\sigma}$ denote the $K$-algebraic group stabilizing the hermitian form $\langle \cdot, \cdot \rangle_{\sigma}.$  Let $U_{\sigma} = \mathbf{U}_{\sigma}(K).$  and $U = U_1.$

\begin{defn}
An \emph{$O_L$-lattice} $M \subset (L^{\oplus n}, \langle \cdot, \cdot \rangle )$ is a submodule locally free of rank $n$ for which $\langle m_1,m_2 \rangle \in O_L$ for all $m_1,m_2 \in M.$

Let $M \subset (L^{\oplus n}, \langle \cdot, \cdot \rangle )$ be an $O_L$-lattice.  Its \emph{dual lattice $M^\vee$} is
$$M^\vee := \{ x \in L^{\oplus n}: \langle x,m \rangle \in O_L \text{ for all } m \in M \}.$$   

An $O_L$-lattice $M \subset L^{\oplus n}$ is \emph{unimodular} if $M = M^\vee.$ 
\end{defn}

\begin{prop}
There is a bijection between orbits of $GL_n(O_L)$ acting on $GL_n(O_L)^{+},$ the symmetric positive elements of $GL_n(O_L),$ by $\ast$-conjugation and the disjoint union of orbits of $U_{\epsilon}$ acting on free, unimodular lattices $M \subset (L^{\oplus n}, \langle \cdot,\cdot \rangle_{\epsilon}), \sigma \in \Sigma.$   
\end{prop}

\begin{proof}
Let $A \in GL_n(O_L)$ be symmetric and positive definite.  Define
$$\langle z, w \rangle ' := \langle A z, w \rangle \text{ for all } z, w \in L^{\oplus n}.$$


Note that $\det ( \langle \cdot,\cdot \rangle' ) = \det( \langle \cdot ,\cdot \rangle_{\sigma_A})$ for a unique $\sigma_A \in \Sigma.$  
Therefore,
$$(L_v^{\oplus n}, \langle \cdot , \cdot \rangle' ) \cong (L_v^{\oplus n}, \langle \cdot,\cdot \rangle_{\sigma_A})$$
for all non-archimedean places $v$ of $K$ \cite[Theorem 3.1]{Jacobowitz}.  Because $\langle \cdot,\cdot \rangle$ is positive definite, $(L_v^{\oplus n}, \langle \cdot,\cdot \rangle_{\sigma,v}) \cong ( L_v^{\oplus n}, \langle \cdot,\cdot \rangle'_v)$ for all archimedean places $v$ of $K.$  

By the Hasse principle for hermitian forms \cite[\S 10, Theorem 1.1]{Scharlau}, $(L^{\oplus n} \langle \cdot,\cdot \rangle_\sigma) \cong (L^{\oplus n},\langle \cdot,\cdot \rangle').$  Let $\varphi_A$ be one choice of isometry.   Define
\begin{align*}
\Phi: GL_n(O_L)^{+} &\rightarrow \bigsqcup_{\sigma \in \Sigma} \{ \text{free unimodular lattices in } (L^{\oplus n}, \langle \cdot,\cdot \rangle_{\sigma}) \text{ of determinant } \sigma \}\\
A &\mapsto (\varphi_A(O_L^{\oplus n}), \langle \cdot, \cdot \rangle_{\sigma_A}). 
\end{align*} 
Note that $\Phi(A)$ is independent of choice of $\varphi_A$ modulo the action of $U_{\sigma_A}$ and $\sigma_A$ is constant along $\ast$-conjugation orbits.  The map $\Phi$ descends to a bijection
\begin{equation} \label{unimodularbijection}
\Phi: \frac{GL_n(O_L)^{+} }{\ast\text{-conjugation}} \xrightarrow{\sim} \bigsqcup_{\sigma \in \Sigma} \frac{\{ \text{free unimodular lattices in } (L^{\oplus n}, \langle \cdot,\cdot \rangle_{\sigma}) \text{ of determinant } \sigma \}}{U_{\sigma}\text{-translation}}.
\end{equation}
The inverse bijection maps a free unimodular lattice to its Gram matrix with respect to an arbitrary choice of basis.
\end{proof}

\subsection{Local equivalence and class groups} \label{localequivalenceclassgroups}
\begin{defn}
For any $O_{L_v}$ lattice $M \subset (L_v^{\oplus n}, \langle \cdot,\cdot \rangle_{\sigma})$ define its \emph{norm ideal} $nM$ to be the $O_{L_v}$ ideal generated by $\langle m,m \rangle_{\sigma}$ for all $m \in M.$
\end{defn}

\begin{prop} \label{unimodularequivalence}
Any two unimoduar lattices $M,M'$ in $(L_v^{\oplus n}, \langle \cdot, \cdot \rangle_{\sigma})$ of fixed determinant $\sigma$ and for which $nM=nM'$ are equivalent under $\mathbf{U}_{\sigma}(K_v).$
\end{prop}

\begin{proof}
This follows from \cite[Proposition 10.4]{Jacobowitz}.
\end{proof}

For unimodular lattices $M, M' \subset (L_v^{\oplus n}, \langle \cdot, \cdot \rangle_{\sigma}),$ the condition $nM = nM'$ is automatically satisfied if $L_v / K_v$ is unramified \cite[Theorem 7.1]{Jacobowitz} or if $L_v / K_v$ is ramified and $v \nmid 2$ \cite[Theorem 8.2]{Jacobowitz}.  In particular, letting $\Sigma' = \Sigma \times \{ \text{ideals } J \text{ dividing } 2 O_L \},$ we may refine right side of the decomposition from \eqref{unimodularbijection} to
\begin{equation} \label{refinedunimodularbijection}
\bigsqcup_{(\sigma,J) \in \Sigma'}
 \frac{\{ \text{free unimodular lattices } M \subset (L^{\oplus n}, \langle \cdot,\cdot \rangle_{\sigma}) \text{ of determinant } \sigma \text{ and } nM = J \}}{U_{\sigma}\text{-translation}}.
\end{equation}

Proposition \ref{unimodularequivalence} tells us that every individual term in \eqref{refinedunimodularbijection} consists only of lattices which are everywhere locally equivalent.  Therefore, for any representative lattice $M_{\sigma,J} \subset (L^{\oplus n}, \langle \cdot, \cdot \rangle_{\sigma}),$ there is a bijection 
\begin{align} \label{genusunimodularlattice}
&\mathbf{U}_{\sigma}(K) \backslash \mathbf{U}_{\sigma}(\finadele_K) / \mathrm{Stab}_{\mathbf{U}_{\sigma}(\finadele_K)}(M_{\sigma,J}) \nonumber \\
&\xrightarrow{\sim}  \frac{\{ \text{ projective unimodular lattices } M \subset (L^{\oplus n}, \langle \cdot,\cdot \rangle_{\sigma}) \text{ of determinant } \sigma \text{ and } nM = J \}}{U_{\sigma}\text{-translation}}.
\end{align}

\subsection{The mass formula for definite unitary groups}

By the Siegel mass formula,
\begin{align} \label{massformula}
\sum_{M \in C_{\sigma}} \frac{1}{\# \mathrm{Aut}(M)} &= \frac{\vol_{\mathrm{Tam}} ( \mathbf{U}_{\sigma}(K) \backslash \mathbf{U}_{\sigma}(\adele_K)) }{\vol_{\mathrm{Tam}} ( \mathrm{Stab}_{\mathbf{U}_{\sigma}(\adele_K)}(M)) } \nonumber \\
&= \frac{2}{\vol_{\mathrm{Tam}} ( \mathrm{Stab}_{\mathbf{U}_{\sigma}(\adele_K)}(M))}.
\end{align}

Though $\mathbf{U}_{\sigma}$ has non-trivial center, its Tamagawa measure is nonetheless finite because its center is anisotropic.  The numerators are equal, in passage to the second line, because the Tamagawa volume of $\mathbf{U}_{\sigma}(K) \backslash \mathbf{U}_{\sigma}(\adele_K)$ equals 2 \cite[\S 10.9]{GY}.    

\begin{rem}
The objects $M$ on the left side of \eqref{massformula} enumerate isomorphism classes of \emph{projective} hermitian lattices, not only free ones.  A priori, the weighted count on the left side of \eqref{massformula} serves as an upper bound for the same sum running over \emph{free} hermitian lattices.  This upper bound is a lower bound too when  $\mathrm{Cl}(L) = 0.$  
\end{rem}

Gan and Yu compute the volume of the adelic stabilizer  \cite[Theorem 10.20]{GY}.  For self-dual hermitian lattices $M,$ their result simplifies to:

\begin{equation} \label{ganyuvolume}
\frac{1}{\vol_{\mathrm{Tam}} ( \mathrm{Stab}_{\mathbf{U}_{\sigma}(\adele_K)}(M))} = c(M) \cdot \frac{|D_{K/\mathbb{Q}}|^{\frac{n^2}{2}}}{\prod_{v \text{ finite}} \beta_{M_v}}, 
\end{equation}

where

\begin{equation*}
c(M) := \left( \prod_{d=1}^n \frac{(d-1)!}{(2\pi)^d} \right)^{[K:\mathbb{Q}]} |\mathrm{Norm}_{K/\mathbb{Q}}(D_{L/K})|^{\frac{n(n+1)}{4}}, 
\end{equation*}

\begin{equation*}
\beta_{M_v} := \begin{cases}
\frac{\# \mathbf{U}_\sigma(k_v)}{q_v^{\dim \mathbf{U}_{\sigma}} } & \text{ if } L_v / K_v \text{ is unramified } \\
\frac{\# \mathcal{G}_v(k_v)}{q_v^{\dim \mathcal{G}_v}} & \text{ if } L_v / K_v \text{ is ramified},
\end{cases}
\end{equation*}

where $\mathcal{G} / k_v $ is one of the two possible orthogonal groups over $k_v.$  The product in \eqref{ganyuvolume} converges conditionally because the center of $\mathbf{U}_{\sigma}$ is anisotropic over $K.$  We next show that the product of factorials in $c(M)$ is much larger than all other terms appearing in \eqref{ganyuvolume}.    

\subsubsection{Bounding discriminants in \eqref{ganyuvolume}}
By assumption, $K$ is the fraction field of some $p$-Weil number $w = p^{\frac{1}{2}}e^{i \theta_0}.$  Therefore, 

\begin{equation} \label{firstdiscriminantbound}
\log \left( |D_{K/\mathbb{Q}}|^{\frac{n^2}{2}} \right) \leq \frac{n^2}{2} \log \left( |\mathrm{Disc}_{\mathbb{Z}[w]/\mathbb{Z}}| \right) = O(n^2 [K:\mathbb{Q}]^2).
\end{equation}

The discriminant of the ring extension $O_K[w] / O_K$ equals $p \sin^2 \theta_0.$  Let $w_j = p^{\frac{1}{2}} e^{i \theta_j}$ be the conjugates of $w.$ Thus,

\begin{equation} \label{normrelativediscriminant}
| \mathrm{Norm}_{K/\mathbb{Q}}(D_{L/K}) | \leq | \mathrm{Norm}_{K/\mathbb{Q}}(\mathrm{Disc}_{O_K[w]/O_K}) | =    p^{[K:\mathbb{Q}]} \prod \sin^2 \theta_j  
\end{equation}

whose logarithm is $O([K : \mathbb{Q}]).$  Therefore,  

\begin{equation} \label{seconddiscriminantbound}
\log \left( | \mathrm{Norm}_{K/\mathbb{Q}}(D_{L/K}) |^{\frac{n(n+1)}{4}} \right) = O \left(n^2 [K:\mathbb{Q}] \right).
\end{equation}

\subsubsection{Bounding the local densities $\beta_{M_v}$ in \eqref{ganyuvolume}}

Let $S = \mathrm{Ram}(L/K).$  Let $E$ denote the set of exponents of the rank $n$ orthogonal group:
$$E = \begin{cases}
\{ 2,4,\ldots, n-1 \} & \text{ if } n \text{ is odd} \\
\{ 2,4, \ldots n, \frac{n}{2} \} & \text{ if } n \text{ is even}.
\end{cases}$$

Suppose $n > 2.$  Then
\begin{align} \label{badfactorsbound}
\prod_{v \in S} \frac{1}{\beta_{M_v}} &\leq 2^{|S|}\prod_{v \in S} \prod_{e \in E} \frac{1}{(1 - q_v^{-e})} \nonumber \\
&\leq 2^{|S|} \prod_{e \in E} \zeta_K(e) \nonumber \\
&\leq 2^{|S|} \zeta_K(2)^{\frac{n}{2} + 1} \nonumber \\
& \leq 2^{|S|} \zeta_{\mathbb{Q}}(2)^{(\frac{n}{2}+1)[K : \mathbb{Q}]}.
\end{align}

By \eqref{normrelativediscriminant}, $| \mathrm{Norm}_{K/\mathbb{Q}}(D_{L/K}) | = \exp(O([K:\mathbb{Q}])).$  Since the product of $m$ distinct primes is at least $m! \geq \exp(m \log m),$ at most $O([K:\mathbb{Q}])$ rational primes divide $\mathrm{Norm}_{K/\mathbb{Q}}(D_{L/K}).$  Every such prime lies under at most $[K:\mathbb{Q}]$ primes of $O_K.$  Therefore 
\begin{equation} \label{numberofbadprimes}
|S| = O([K : \mathbb{Q}]^2).
\end{equation}
Thus \eqref{badfactorsbound} and \eqref{numberofbadprimes} give
\begin{equation} \label{rankgreaterthan2badfactors}
\prod_{v \in S} \frac{1}{\beta_{M_v}} = \exp(O(\max \{ n[K:\mathbb{Q}], [K:\mathbb{Q}]^2 \})) \text{ if } n > 2.
\end{equation}

Suppose $n = 2.$  Then by \eqref{numberofbadprimes},
\begin{equation} \label{rankequals2badfactors}
\prod_{v \in S} \frac{1}{\beta_{M_v}} = 2^{|S|} \prod_{v \in S} \begin{cases}
\frac{1}{1 - q_v^{-1}} \\
\frac{1}{1 +q_v^{-1}}
\end{cases}
\leq 2^{|S|} \prod_{v \in S}\frac{1}{1 - q_v^{-1}} \leq 2^{|S|} \left( \frac{1}{1-\frac{1}{2}} \right)^{|S|} = \exp(O([K:\mathbb{Q}]^2)) \text{ if } n = 2.
\end{equation}

Suppose finally that $v \notin S.$  There is an exact sequence of smooth algebraic groups over $O_{K,v}$
$$1 \rightarrow \mathbf{SU}_{\sigma} \rightarrow \mathbf{U}_{\sigma} \rightarrow \mathbf{U}_h \rightarrow 1,$$  
where $h$ is the unitary form on the rank 1 $O_{L_v}$-module $\wedge^n O_{L_v}^{\oplus n}$ with $h(\wedge e_i, \wedge f_j) = \det(\langle e_i,f_j\rangle)_\sigma.$  In particular, since $\mathbf{SU}_\sigma / k_v$ is connected, the sequence
$$1 \rightarrow \mathbf{SU}_{\sigma}(k_v) \rightarrow \mathbf{U}_\sigma(k_v) \rightarrow \mathbf{U}_h(k_v) \rightarrow 1$$
is exact by Lang's theorem.  So

\begin{equation}
\frac{\# \mathbf{U}_\sigma(k_v)}{q_v^{\dim \mathbf{U}_\sigma}} = \frac{\# \mathbf{SU}_\sigma(k_v)}{q_v^{\dim \mathbf{SU}_\sigma}} \cdot \frac{\# \mathbf{U}_h(k_v)}{q_v^{\dim \mathbf{U}_h}}. 
\end{equation}

Therefore,

\begin{align} \label{upperboundgoodfactors}
\prod_{v \notin S} \frac{1}{\beta_{M_v}} &= \prod_{v \notin S} \left( \frac{\# \mathbf{SU}_\sigma(k_v)}{q_v^{\dim \mathbf{SU}_\sigma}} \right)^{-1} \cdot \prod_{v \notin S} \left( \frac{\# \mathbf{U}_h(k_v)}{q_v^{\dim \mathbf{U}_h}} \right)^{-1} \nonumber \\
&=  \prod_{v \notin S} \left( \frac{\# \mathbf{SU}_\sigma(k_v)}{q_v^{\dim \mathbf{SU}_\sigma}} \right)^{-1} \cdot L(1, \chi_{L/K}) \nonumber \\
&\leq \prod_{e = 2}^n \zeta_K(e) \cdot L(1,\chi_{L/K}) \nonumber \\
&\leq \prod_{e = 2}^n \zeta_\mathbb{Q}(e)^{[K:\mathbb{Q}]} \cdot L(1, \chi_{L/K}) \nonumber \\
&\leq \zeta_{\mathbb{Q}}(2)^{n [K:\mathbb{Q}]} \cdot L(1, \chi_{L/K}),
\end{align}

where $\chi_{L/K}$ is the quadratic Artin character corresponding to the field extension $L/K.$  Since $\zeta_L(s) = \zeta_K(s) L(s, \chi_{L/K})$

\begin{align} \label{L1chiupperbound}
\log L(1,\chi_{L/K}) &= \log \left( \mathrm{Res}_{s = 1}\zeta_L(s) \right) - \log \left( \mathrm{Res}_{s = 1}\zeta_K(s) \right) \nonumber \\
&= O([K : \mathbb{Q}] \log [K : \mathbb{Q}]),
\end{align}

where we've used the bound \eqref{zetaresidueupperbound} to reach the second line.  Combining \eqref{upperboundgoodfactors}, \eqref{L1chiupperbound}, and \eqref{rankgreaterthan2badfactors} or \eqref{rankequals2badfactors} (depending on whether $n > 2$ or $n = 2$ yields

\begin{equation} \label{finalupperboundlocaldensities}
\prod_v \frac{1}{\beta_v} = \exp( O (n [K: \mathbb{Q}]^2 ) ).
\end{equation}

\subsubsection{Final estimate for the mass of a unimodular hermitian lattice}
Combining \eqref{massformula}, \eqref{ganyuvolume}, \eqref{firstdiscriminantbound}, \eqref{seconddiscriminantbound}, and \eqref{finalupperboundlocaldensities}:

\begin{equation} \label{finalmassestimate}
\log \left( \sum_{M \in C_{\sigma}} \frac{1}{\# \mathrm{Aut}(M)} \right) = [K : \mathbb{Q}] n^2 \log n + O(n^2[K:\mathbb{Q}]^2).
\end{equation}

\subsubsection{Upper bounds on the order of automorphism groups} 

\begin{lem}[Minkowski, Schur, Serre \cite{Serre}] \label{finitesubgroupGL}
The maximal order of a finite subgroup of $GL_n(\mathbb{Z})$ is $\exp(n \log n + O(n)).$ 
\end{lem}

\begin{proof}
Let $A \subset GL_n(\mathbb{Z})$ be any finite group.  Since $\Gamma(q) = \{ g \in GL_n(\mathbb{Z}): g \equiv 1 (q) \}$ is torsion free for all primes $q \geq 3,$ the reduction mod $q$ map from $A$ into $GL_n(\mathbb{Z} / q\mathbb{Z})$ is injective.  Therefore, 
\begin{align*}
|A| &\text{ divides } \# GL_n(\mathbb{Z} / q\mathbb{Z}) = q^{\frac{n(n-1)}{2}} \prod_{i = 1}^n (q^i - 1) \\
\implies |A| &\text{ divides } d:= \gcd_{\text{primes } q \geq 3} \left(  q^{\frac{n(n-1)}{2}} \prod_{i = 1}^n (q^i - 1)  \right).
\end{align*}
Let $p$ be odd.  The group $\mathbb{Z}_p^{\times}$ is procyclic, say with generator $t.$  By Dirichlet's theorem on primes in arithmetic progressions, we may choose a prime $q_0$ which is congruent to $t$ modulo a very high power of $p.$  Then 
$$ \mathrm{val}_p(q_0^i - 1) = \begin{cases}
\mathrm{val}_p(i) + 1 & \text{ if } i \equiv 0 \; (p-1) \\
0 & \text{ otherwise}.
\end{cases}$$
Summing over $1 \leq i \leq n$ gives \cite[Theorem 1]{Serre}
$$\mathrm{val}_p(d) \leq M(n,p) := \sum_{a \geq 0} \left\lfloor \frac{n}{(p-1) p^a} \right\rfloor \leq \frac{np}{(p-1)^2} = \frac{n}{p} + O \left(\frac{n}{p^2} \right).$$
Likewise, we readily check that 
$$\mathrm{val}_2(d) \leq \mathrm{val}_2( \# GL_n(\mathbb{Z} / 3\mathbb{Z}) ) = O(n).$$
Therefore, 
\begin{align*}
\log |A| & \leq \log d \\
& \leq \sum_{\text{primes } p \leq n + 1} M(n,p) \log p + \sum_{\text{primes } p > n+1} M(n,p) \log p \\
& \leq n \sum_{\text{primes } p \leq n + 1} \left( \frac{1}{p} + O\left(\frac{1}{p^2} \right) \right) \log p + \sum_{\text{primes } p > n+1} 0 \log p \\
&= n \log n + O(n).
\end{align*}
This upper bound is achieved by the signed permutation matrix subgroup of $GL_n(\mathbb{Z}),$ which has order $n! 2^n.$ 
\end{proof}

Let $M$ be any hermitian projective $O_L$-lattice of rank $n.$  Forgetting the hermitian and $O_L$-structure induces an embedding
$$\mathrm{Aut}(M) \hookrightarrow GL_{[L:\mathbb{Q}] n}(\mathbb{Z}).$$
By Lemma \ref{finitesubgroupGL},
\begin{equation} \label{largestfinitesubgroup}
\log \left( \# \mathrm{Aut}(M) \right) \leq [L:\mathbb{Q}] n \log \left( [L:\mathbb{Q}] n \right) + O([L : \mathbb{Q}] n)= O([K:\mathbb{Q}]^2 n^2).
\end{equation}

The estimates \eqref{largestfinitesubgroup} and \eqref{finalmassestimate} imply
\begin{align} \label{finalestimateclassnumber}
\log \left( \# C_{\sigma} \right) &= \log \left( \sum_{M \in C_{\sigma}} \frac{1}{\# \mathrm{Aut}(M)} \right) + O \left( \log \left( \max_{M \in C_{\sigma}} \# \mathrm{Aut}(M)  \right) \right) \nonumber \\
&= [K : \mathbb{Q}] n^2 \log n + O([K : \mathbb{Q}]^2 n^2).
\end{align}

The discussion of \S \ref{astconjugationorbitsunimodularhermitianlattices}  and \S \ref{localequivalenceclassgroups}  shows that the collection of $\ast$-conjugation orbits on $R_{\{L\}, \{n\}}$ injects into a disjoint union of sets $C_{\sigma,J}$ indexed by $\Sigma' = \Sigma \times \{ \text{ideals } J \text{ dividing } 2 O_L \}.$  Every $C_{\sigma,J}$ has logarithmic size $ [K : \mathbb{Q}] n^2 \log n + O([K : \mathbb{Q}]^2 n^2)$ by \eqref{finalestimateclassnumber}.  The number of ideals in $O_L$ dividing $2O_L$ is at most $\exp(O([K:\mathbb{Q}])).$  The size of $\Sigma$ is at most $\exp(O([K : \mathbb{Q}]^2))$ by \eqref{boundingpolarizationscommutative}. Therefore,
\begin{equation} \label{finalestimateastconjugationorbits}
\log \# \left( \ast\text{-conjugation orbits of $R_{\{L\}, \{n\}}^\times$ on } R_{\{L\}, \{n\}} \right) \leq [K : \mathbb{Q}] n^2 \log n + O([K : \mathbb{Q}]^2 n^2)
\end{equation}
Equality holds if the class group of $O_L$ is trivial.

\subsubsection{Final estimate for the number of $\ast$-conjugation orbits on $R_{\mathbf{L},\mathbf{n}}$}
Summing the estimate \eqref{finalestimateastconjugationorbits} over all factors of $R_{\mathbf{L},\mathbf{n}},$ the number of $\ast$-conjugation orbits of $R_{\mathbf{L},\mathbf{n}}^\times$ on the positive symmetric elements of $R_{\mathbf{L},\mathbf{n}}^\times$ is at most
\begin{align} \label{finalestimateclassnumberproduct}
& \sum_{i=1}^k [K_i: \mathbb{Q}] n_i^2 \log n_i + \sum_{i = 1}^k O([K_i: \mathbb{Q}]^2 n_i^2) \nonumber \\
&=  \sum_{i=1}^k [K_i: \mathbb{Q}] n_i^2 \log n_i + O(g^2),
\end{align}

where we've used $\sum_{i=1}^k [K_i:\mathbb{Q}] n_i = g.$  Equality holds if the class group of every $O_{L_i}$ is trivial.

\subsection{Consequences of orbit counts for abelian variety statistics}
\label{orbitcountconsequences}

\begin{lem} \label{normconvexity}
Fix $\frac{1}{2}  < \epsilon < 1.$  Suppose $0 \leq x_i \leq \epsilon g, \ldots 0 \leq x_k \leq \epsilon g$ and $\sum_{i = 1}^k x_i = g.$  Then  
$$||x||^2 = \sum_{i = 1}^k x_i^2 \leq g^2 \left( \epsilon^2 + (1 - \epsilon)^2 \right) = g^2 \left(1 - 2\epsilon(1-\epsilon)\right).$$ 
\end{lem}

\begin{proof}
The function $||x||^2$ is strictly convex.  The maximum must occur at one of the $k(k+1)$ extreme points of the polyhedron $[0,\epsilon]^k \cap \{ x: \sum x_i = g \}.$  All such extreme points $y$ satisfy $||y||^2 = g^2\left(\epsilon^2 + (1 - \epsilon)^2 \right).$
\end{proof}

\begin{prop} \label{astconjugationorbits}
Suppose $n_{\mathbf{L},\mathbf{n}},$ the number of $\ast$-conjugation orbits of $R_{\mathbf{L},\mathbf{n}}^\times$ on the positive symmetric elements of $R_{\mathbf{L},\mathbf{n}}^\times,$ satisfies
$$\log n_{\mathbf{L}, \mathbf{n}} \geq 0.99 g^2 \log g.$$
Then there is some $i_0$ for which $n_{i_0} \geq 0.99 g$ and $[K_{i_0}:\mathbb{Q}] = 1.$
\end{prop}

\begin{proof}
Define $x_i := [K_i:\mathbb{Q}]n_i.$  Suppose $x_i < 0.99g$ for all $1 \leq i \leq k.$  By Lemma \ref{normconvexity},
\begin{align*}
\sum_{i = 1}^k [K_i: \mathbb{Q}] n_i^2 \log n_i &\leq \sum_{i = 1}^k [K_i: \mathbb{Q}]^2 n_i^2 \log g \\
&= \sum_{i = 1}^k x_i^2 \log g \\
&\leq (1 - 2 \cdot 0.99 \cdot 0.01) g^2 \log g \\
&\approx 0.98 g^2 \log g. 
\end{align*}

Therefore, if $\log n_{\mathbf{L},\mathbf{n}} \geq 0.99 g^2 \log g,$ we must have $x_{i_0} = [K_{i_0}:\mathbb{Q}] n_{i_0} \geq 0.99 g$ for some $i_0.$  In particular,
\begin{align*}
\sum_{i \neq i_0} [K_i:\mathbb{Q}]n_i^2 \log n_i &\leq \sum_{i \neq i_0} \left( [K_i:\mathbb{Q}]n_i \right) \cdot n_i \log g \\
&\leq \sum_{i \neq i_0} 0.01 g \cdot n_i \cdot \log g \\
&\leq 0.01 g^2 \log g.
\end{align*}

The main contribution must therefore come from
$$\left( [K_{i_0}: \mathbb{Q}]n_{i_0} \right) n_{i_0} \log n_{i_0} \approx g n_{i_0} \log n_{i_0}.$$
If $[K_{i_0}:\mathbb{Q}] \geq 2,$ then $n_{i_0} \leq \frac{g}{2}$ and the above is at most $\approx \frac{1}{2} g^2 \log g.$  Therefore,
$$[K_{i_0}:\mathbb{Q}] = 1 \text{ and } n_{i_0} \geq 0.99 g.$$
\end{proof}

\begin{lem} \label{goodellipticcurves}
For some density $\geq  1 - 2^{-9}$ subset of the primes, there exists an elliptic curve $E / \mathbb{F}_p$ for which the number of principal polarizations on $E^g$ equals $\exp( g^2 \log g + O(g^2)).$
\end{lem}

\begin{proof}
Let $L = \mathbb{Q}(\sqrt{-d})$ for a squarefree positive integer $d.$  Suppose $O_L$ has class number 1.  Let $p$ be a prime and $f(x) = x^2 - ax + p$ be a $p$-Weil polynomial.  

Suppose we can solve the equation $a^2 - 4p = -d b^2$ for some $b \in \mathbb{Q}.$  Then $\mathbb{Z}[x] / (f(x))$ is an order in $O_L.$  According to \cite[Theorem 6.1]{Waterhouse}, there is an elliptic curve $E / \mathbb{F}_p$ satisfying $\End(E) = O_L.$  In particular, $\ast$-conjugation orbits of $GL_n(O_L)$ on the symmetric positive elements of $GL_n(O_L)$ are in bijection with principal polarizations of $E^g.$  Because $\mathrm{Cl}(O_L) = 0,$ the upper bound from \eqref{finalestimateclassnumberproduct} is an equality and the conclusion of the Lemma holds:
\begin{equation} \label{productellipticcurve}
\log \left( \# \{ \text{principal polarizations on } E^g \} \right) = g^2 \log g + O(g^2).
\end{equation}
Since $2O_L \subset \mathbb{Z}[\sqrt{-d}], 4p$ is the norm of an element of $\mathbb{Z}[\sqrt{-d}]$ if $p$ is the norm of an element of $O_L,$ i.e. if $p$ splits in $O_L.$  For such primes $p, a^2 - 4p = -db^2$ has a solution for \emph{integers} $a,b.$  By quadratic reciprocity, the proportion of primes $p$ which are not split in any of the 9 imaginary quadratic fields of class number 1 is $\frac{1}{2^9}.$
\end{proof}

\begin{rem}
We strongly suspect that the conclusion of Lemma \ref{goodellipticcurves} holds for all primes $p.$  The class number 1 hypothesis in the proof is only used to guarantee that the upper bound \eqref{finalestimateclassnumberproduct} be an equality. 
\end{rem}

\begin{cor} \label{almosteverythingproductellipticcurve}
Let $p$ be a prime satisfying the conclusion of Lemma \ref{goodellipticcurves}.  Conjecture \ref{mrpiconjecture} implies that the proportion of principally polarized abelian varieties over $\mathbb{F}_p$ for which $A$ admits an isogeny factor $E^h$ for some elliptic curve $E$ and some $h \geq 0.99g$ approaches 1 as $g$ grows.
\end{cor}

\begin{proof}
Call an abelian variety \emph{bad} if it does not admit an isogeny factor as large as stipulated in the Corollary statement.  If an isomorphism class $A$ of abelian varieties is bad, then Conjecture \ref{mrpiconjecture} and Proposition \ref{astconjugationorbits} imply that
$$n_A \leq \exp(0.99 g^2 \log g + O(g^2)),$$
where $n_A$ is the number of principal polarizations on $A.$  But \eqref{finalupperboundisomorphismclasses} and Corollary \ref{upperboundisogenyclasses} imply that the total number of isomorphism classes of abelian varieties over $\mathbb{F}_p$ is at most $\exp(O(g^2)).$  Therefore, the total number of bad abelian varieties over $\mathbb{F}_p$ is at most 
$$\exp(O(g^2)) \exp(0.99g^2 \log g + O(g^2)) = \exp(0.99g^2 \log g + O(g^2)).$$
On the other hand, the total number of principal polarizations on $E^g$ equals $\exp(g^2 \log g + O(g^2)).$  Therefore,
\begin{align*}
\text{proportion of PPAVs which are bad} &= \frac{\text{number of bad PPAVs}}{\text{total number of PPAVs}} \\
&\leq \frac{\exp(0.99g^2 \log g + O(g^2))}{\exp(g^2 \log g + O(g^2))} \\
&\leq \exp(- 0.01 g^2 \log g + O(g^2)),
\end{align*}
which rapidly approaches 0 as $g$ grows.
\end{proof}

\subsection{Shortcomings of archimedean random matrix heuristics for modelling families of abelian varieties}
\label{archimedeanrandommatrix}

\begin{cor}
Let $p$ be any prime satisfying the conclusion of Lemma \ref{goodellipticcurves}.  If Conjecture \ref{mrpiconjecture} is true, then the discrepancy between the normalized spacing measure $\mu_g = \mu(A_g / \mathbb{F}_p)$ for $A_g / \mathbb{F}_p$ and the GUE normalized spacing measure $\mu_{\mathrm{GUE}}$ is $\geq 0.99.$
\end{cor}

\begin{proof}
Suppose $A / \mathbb{F}_p$ is a $g$-dimensional principally polarized abelian variety.  By Corollary \ref{almosteverythingproductellipticcurve}, the probability that $A$ contains an isogeny factor of the form $E^h$ for some $h \geq 0.99g$ approaches 1 as $g$ grows.  If $z(E), \overline{z}(E)$ are roots of the Frobenius characteristic polynomial of $E,$ then $z, \overline{z}$ both occur with multiplicity $\geq h$ in the Frobenius characteristic polynomial of $A.$  But then for $\alpha = h / g \geq 0.99,$ the normalized spacing measure is $\alpha \delta_0 + $ weighted sum of point masses of total mass $1 - \alpha.$  Therefore,
$$\mathrm{discrep}(\mu_g, \mu_{\mathrm{GUE}}) \geq 0.99.$$  
\end{proof}

\begin{rem}
Note that
$$\lim_{p \rightarrow \infty} \lim_{g \rightarrow \infty} \mathrm{discrep}(\mu(A_g / \mathbb{F}_p), \mu_{\mathrm{GUE}}) = 0$$ 
would be strongly analogous to conjectures of Katz and Sarnak \cite{KatzSarnak} on the normalized spacing measures for families of curves with large geometric monodromy.  We emphasize, however, that Katz and Sarnak never posited their conjecture for families of abelian varieties.  
\end{rem}

\subsection{Shortcomings of $p$-adic random matrix heuristics for modelling families of abelian varieties}
\label{padicrandommatrix}

\subsubsection{Cohen-Lenstra heuristics}
The Cohen-Lenstra probability distribution $\mathrm{CL}_\ell$ for finite abelian $\ell$ groups assigns mass proporitional to $\frac{1}{\# \mathrm{Aut}(H)}$ to every finite abelian $\ell$-group $H.$  This is the limiting distribution, as $g \rightarrow \infty,$ of $\mathrm{coker}(1 - F)$ where $F \in \End_g(\mathbb{Z}_\ell)$ is sampled uniformly with respect to Haar measure \cite{FW}.  For many families $\mathcal{F}_g$ of abelian varieties $A /\mathbb{F}_p,$ such as Jacobians of genus $g$ hyperelliptic curves \cite{EVW}, the Cohen-Lenstra distribution is expected to accurately model $A(\mathbb{F}_p)_\ell$ for randomly sampled $A \in \mathcal{F}_g$ as $g$ grows, at least for $\ell \neq p.$  Randomly sampled $F$ in $\End_g(\mathbb{Z}_\ell)$ is supposed to model the Frobenius endomorphism of $T_{\ell}(A)$ for $A \in \mathcal{F}_g$ for $g$ large.

For any finite set of primes $S,$ the Cohen-Lenstra distribution $\mathrm{CL}_S$ for finite abelian groups supported at $S$ assigns mass $\prod_{v \in S} \mathrm{CL}_v(H_v),$ where $H_v$ is the $v$-primary part of $H.$  The distribution $\mathrm{CL}_S$ is expected to model $A(\mathbb{F}_p)_S$ for randomly sampled $A \in \mathcal{F}_g$ as $g$ grows, provided $p \notin S.$.  The added content is that $A(\mathbb{F}_p)_\ell$ and $A(\mathbb{F}_p)_{\ell'}$ are expected to be independent for distinct primes $\ell, \ell' \neq p.$     

\subsubsection{Enhanced Cohen-Lenstra heuristics}
We propose a modest enhancement of Cohen-Lenstra distribution to model the joint distribution of $A(\mathbb{F}_{p^{n_1}})_S, \ldots, A(\mathbb{F}_{p^{n_k}})_S$ for any finite collection $\mathbf{n} = (n_1,\ldots, n_k)$ of distinct positive integers $n_i.$
\begin{defn}
Let $\mathbf{n} = (n_1,\ldots,n_k)$ be a tuple of distinct positive integers.  We define the $\mathbf{n}$-\emph{Cohen-Lenstra distribution} by
$$\mathrm{CL}_{\mathbf{n},\ell}(H_1,\ldots,H_k) := \lim_{g \to \infty} \mathrm{Haar}_g \left( F \in \End_g(\mathbb{Z}_\ell):  \mathrm{coker}(1 - F^{n_j}) \cong H_j,1 \leq j \leq k \right),$$
where $\mathrm{Haar}_g$ is the volume 1 Haar measure on $\End_g(\mathbb{Z}_\ell).$  For any finite collection of primes $S,$ define
$$\mathrm{CL}_{\mathbf{n},S}(H_1,\ldots,H_k) := \prod_{v \in S} \mathrm{CL}_{\mathbf{n},v}((H_1)_v,\ldots,(H_k)_v).$$
\end{defn}

\begin{exam}
Let $\ell$ be prime.  Then
\begin{align*}
\mathrm{CL}_{(1,2), \ell}(0,0) &= \lim_{g \to \infty} \mathrm{Haar}_g(F: \mathrm{coker}(1 - F) = 0, \mathrm{coker}(1 - F^2) = 0 ) \\
&= \lim_{g \to \infty} \mathrm{Haar}_g(F: 1 - F^2 \text{ invertible } \mod \ell).
\end{align*} 
Upon reduction mod $\ell, \mathrm{Haar}_g$ induces the uniform measure on $\End_g(\mathbb{F}_\ell).$  
\begin{itemize}
\item
Suppose $\ell = 2.$  Then $1 - F^2 = (1 - F)^2 \mod 2.$  The probability that $F \mod 2$ does not have $1$ as an eigenvalue equals
$$\frac{\# GL_g(\mathbb{F}_2)}{\# \End_g(\mathbb{F}_2)} \xrightarrow{g \to \infty} \prod_{n = 1}^{\infty} (1 - 2^{-n}).$$

\item
Suppose $\ell \neq 2.$  The proability that $F \in \End_g(\mathbb{F}_\ell)$ does not have $\pm 1$ as an eigenvalue is uniformly bounded below across all $g.$  Indeed, 
\begin{align*}
&\mathrm{Prob}(F \text{ does not have eigenvalue } \pm 1)\\
&= 1 - \mathrm{Prob}(F \text{ has eigenvalue } \pm 1) \\
&\geq 1 - \mathrm{Prob}(F \text{ has eigenvalue } 1) - \mathrm{Prob}(F \text{ has eigenvalue } -1) \\
&= 1 - 2 \frac{\ell^{g^2} - |GL_g(\mathbb{F}_\ell)|}{\ell^{g^2}}\\
&\geq 1 - 2 \left( 1 - \prod_{k = 1}^\infty (1 - \ell^{-k}) \right) \\
&\geq   1 - 2 \left( 1 - \prod_{k = 1}^\infty (1 - 3^{-k}) \right) \\
&= 0.12025 \cdots \\
&> 0.
\end{align*} 
\end{itemize}
\end{exam}

\begin{conj}
Let $\mathbf{n} = (n_1,\ldots,n_k)$ be a tuple of distinct positive integers and let $S$ be any finite collection of primes not containing $p.$  For any geometric family $\mathcal{F}_g$ of abelian varieties over $\mathbb{F}_p$ for which the Cohen-Lenstra distribution $\mathrm{CL}_S$ accurately models the distribution of $A(\mathbb{F}_p)_S$ for randomly sampled $A \in \mathcal{F}_g$ as $g$ grows, the $\mathbf{n}$-Cohen-Lenstra distribution $\mathrm{CL}_{\mathbf{n},S}$ accurately models the joint distribution of $A(\mathbb{F}_{p^{n_1}})_S, \ldots, A(\mathbb{F}_{p^{n_k}})_S$ for randomly sampled $A \in \mathcal{F}_g$ as $g$ grows.
\end{conj}

\subsubsection{Randomly sampled principally polarized abelian varieties fail enhanced Cohen-Lenstra}
We'll show that the abundance of elliptic curve isogeny factors occuring in typical principally polarized abelian varieties is at ends with the enhainced Cohen-Lenstra heuristics.
\begin{lem} \label{primetoptorsion}
Let $E / \mathbb{F}_p$ be an elliptic curve.  Then $E(\mathbb{F}_p)$ and $E(\mathbb{F}_{p^2})$ are not both $p$-groups.
\end{lem}

\begin{proof}
By the Hasse bounds, 
\begin{align*}
\# E(\mathbb{F}_p) &= p + 1 - a \text{ where } | a | \leq 2 \sqrt{p} \\
\# E(\mathbb{F}_{p^2}) &= p^2 + 1 - b \text{ where } |b| \leq 2p.
\end{align*}
If both $E(\mathbb{F}_p)$ and $E(\mathbb{F}_{p^2})$ are $p$-groups, then $a = 1$ and $b = 1.$  But then for any prime $\ell \neq p,$
\begin{equation*}
p = \det( \frob  | T_\ell(E)) = \frac{a^2 - b}{2} = 0,
\end{equation*}
a contradiction.
\end{proof}

\begin{cor}\label{clfails}
Let $p$ be any prime satisfying the conclusion of Lemma \ref{goodellipticcurves}.  If Conjecture \ref{mrpiconjecture} is true, then there is a finite set of primes $S$ not contaiing $p$ for which the enhanced Cohen-Lenstra heuristics do not accurately describe the joint distribution of the $S$-primary part of the finite abelian groups $A(\mathbb{F}_p)$ and $A(\mathbb{F}_{p^2})$  
\end{cor}

\begin{proof}
Let $S$ be the set of all primes less than $p^2 + 1 + 2p$ excluding $p.$  For every elliptic curve $E / \mathbb{F}_p,$ the groups $E(F_{\mathbb{F}_p})$ and $E(\mathbb{F}_{p^2})$ are $S \cup \{ p \}$-primary.  By Lemma \ref{primetoptorsion}, $\# E(\mathbb{F}_p)_v \geq v$ or $\# E(\mathbb{F}_{p^2})_v \geq v$ for some $v \in S.$  

Let $A / \mathbb{F}_p$ be a $g$-dimensional principally polarized abelian variety.  Assuming Conjecture \ref{mrpiconjecture}, Corollary \ref{almosteverythingproductellipticcurve} implies that $A$ contains an isogeny factor $E^h$ for some elliptic curve $E$ and $h \geq 0.99 g$  with probability approaching 1 as $g$ grows.  So with probability approaching $1$ as $g$ grows, $\# A(\mathbb{F}_p)_v \geq v^{0.99g}$ or $\# A(\mathbb{F}_{p^2})_v \geq v^{0.99g}$  for some $v \in S.$  

On the other hand,
\begin{align*}
\mathrm{CL}_{(1,2), S}( H_1 = 0, H_2 = 0 ) &= \prod_{v \in S} \mathrm{CL}_{(1,2),v}( (H_1)_v = 0, (H_2)_v = 0 ) \\
&> 0.
\end{align*}
Contradiction.
\end{proof}


\begin{thebibliography}{10}
\bibitem{Kedlayaetal}
Achter, Erman, Kedlaya, Wood, Zureick-Brown.  \emph{A heuristic for the distribution of point counts for random curves over a finite field}. Philos. Trans. A 373 (2015), no. 2040, 20140310.

\bibitem{Clozel}
Clozel.  \emph{Nombre de points des vari\'{e}t\'{e}s de Shimura sur un corps fini (d'après R. Kottwitz)}.  S\'{e}minaire Bourbaki, Vol. 1992/93. Ast\'{e}risque No. 216 (1993), Exp. No. 766, 4, 121-149.  

\bibitem{Conrad}
Conrad.  \emph{Lecture 2: Abelian varieties}.  Stanford Number Theory Learning Seminar (2011).

\bibitem{DejongKatz} 
De Jong, Katz. \emph{Counting the number of curves over a finite field}.

\bibitem{DM}
Deligne, Mumford.  \emph{The irreducibility of the space of curves of given genus}. Inst. Hautes \'{E}tudes Sci. Publ. Math. No. 36 1969 75-109.


\bibitem{DH}
DiPippo, Howe.  \emph{Real polynomials with all roots on the unit circle and abelian varieties over finite fields}.  J. Number Theory 73 (1998), no. 2, 426-450.

\bibitem{EVW}
Ellenberg, Venkatesh, Westerland.  \emph{Homological stability for Hurwitz spaces and the Cohen-Lenstra conjecture over function fields}.   arXiv:0912.0325v3 [math.NT].

\bibitem{FC}
Faltings, Chai.  \emph{Degeneration of abelian varieties. With an appendix by David Mumford}. Ergebnisse der Mathematik und ihrer Grenzgebiete (3) [Results in Mathematics and Related Areas (3)], 22. Springer-Verlag, Berlin, 1990.

\bibitem{FW}
Friedman, Washington.  \emph{On the distribution of divisor class groups of curves over a finite field}.  Th\'{e}orie des nombres (Quebec, PQ, 1987), 227-239, de Gruyter, Berlin, 1989.

\bibitem{GY}
Gan, Yu.  \emph{Group schemes and local densities}. Duke Math. J. 105 (2000), no. 3, 497-524.

\bibitem{Honda}
Honda.  \emph{Isogeny classes of abelian varieties over finite fields}. J. Math. Soc. Japan 20 1968 83-95. 

\bibitem{Jacobowitz}
Jacobowitz.  \emph{Hermitian forms over local fields}. Amer. J. Math. 84 1962 441-465.

\bibitem{KatzSarnak}
Katz, Sarnak.  \emph{Random matrices, Frobenius eigenvalues, and monodromy}. American Mathematical Society Colloquium Publications, 45. American Mathematical Society, Providence, RI, 1999.

\bibitem{Kottwitz}
Kottwitz.  \emph{Points on some Shimura varieties over finite fields}.  J. Amer. Math. Soc. 5 (1992), no. 2, 373-444. 

\bibitem{Lang}
Lang.  \emph{Divisors and endomorphisms on an abelian variety}. Amer. J. Math. 79 1957 761-777. 

\bibitem{Langlands}
Langlands.  \emph{Modular forms and $\ell$-adic representations}. Modular functions of one variable, II (Proc. Internat. Summer School, Univ. Antwerp, Antwerp, 1972), pp. 361-500. Lecture Notes in Math., Vol. 349, Springer, Berlin, 1973.

\bibitem{Lemm}
Lemmermeyer.  \emph{The ambiguous class number formula revisited}. J. Ramanujan Math. Soc. 28 (2013), no. 4, 415-421.

\bibitem{Louboutin}
Louboutin.  \emph{Explicit bounds for residues of zeta functions, values of $L$-functions at $s=1,$ and relative class numbers}.  J. Number Theory 85 (2000), no. 2, 263-282.


\bibitem{Mumford}
Mumford.  \emph{Abelian varieties}.  Tata Institute of Fundamental Research Studies in Mathematics, No. 5 Published for the Tata Institute of Fundamental Research, Bombay; Oxford University Press, London 1970.

\bibitem{Prasad}
Prasad.  \emph{Volumes of $S$-arithmetic quotients of semisimple groups. With an appendix by Moshe Jarden and the author.} Inst. Hautes \'{E}tudes Sci. Publ. Math. No. 69 (1989), 91-117.

\bibitem{ps}
Prendergast-Smith.  \emph{The cone conjecture for abelian varieties}.  J. Math. Sci. Univ. Tokyo 19 (2012), no. 2, 243-261.

\bibitem{Scharlau}
Scharlau.  \emph{Quadratic and hermitian forms}.  Grundlehren der Mathematischen Wissenschaften [Fundamental Principles of Mathematical Sciences], 270. Springer-Verlag, Berlin, 1985.

\bibitem{Serre}
Serre.  \emph{Bounds for the orders of finite subgroups of $G(k)$}. Group representation theory, 405-450, EPFL Press, Lausanne, 2007.


\bibitem{Skoruppa}
Skoruppa.  \emph{Quick lower bounds for regulators of number fields}. Enseign. Math. (2) 39 (1993), no. 1-2, 137-141. 

\bibitem{Tate1}
Tate.  \emph{Endomorphisms of abelian varieties over finite fields}. Invent. Math. 2 1966 134-144. 

\bibitem{Tate}
Tate.  \emph{Classes d'isog\'{e}nie des vari\'{e}t\'{e}s ab\'{e}liennes sur un corps fini (d'apr\`{e}s T. Honda)}.  S\'{e}minaire Bourbaki. Vol. 1968/69: Expos\'{e}s 347-363, Exp. No. 352, 95-110, Lecture Notes in Math., 175, Springer, Berlin, 1971.

\bibitem{Waterhouse}
Waterhouse.  \emph{Abelian varieties over finite fields}.  Ann. Sci. \'{E}cole Norm. Sup. (4) 2 1969 521-560. 

\bibitem{Yun}
Yun. \emph{Orbital integrals and Dedekind zeta functions}. The legacy of Srinivasa Ramanujan, 399-420, Ramanujan Math. Soc. Lect. Notes Ser., 20, Ramanujan Math. Soc., Mysore, 2013.

\end{thebibliography}
\end{document}